\documentclass{elsarticle}

\usepackage{latexsym,lmodern,graphicx,enumitem,amssymb,amsfonts,amsmath,amsthm,dsfont,color}

\newtheorem{theorem}{Theorem}
\newtheorem{proposition}[theorem]{Proposition}
\newtheorem{lemma}[theorem]{Lemma}

\theoremstyle{definition}
\newtheorem{example}[theorem]{Example}
\newtheorem{definition}[theorem]{Definition}
\newtheorem{remark}[theorem]{Remark}

\numberwithin{equation}{section}
\numberwithin{theorem}{section}

\newcommand{\xmlpar}{\par \noindent}

\begin{document}

\title{Hypercontractivity for Markov semi-groups\tnoteref{t1,t2}}

\tnotetext[t1]{Work supported by the Labex MME-DII funded by ANR, reference ANR-11-LBX-0023-01 and ANR-15-CE40-0020-03 - LSD - Large Stochastic Dynamics, the grant of the Simone and Cino Del Duca Foundation, 
and the grant Research Impulse Award DRI165BZ in the UK.}

\author[1]{C. Roberto\corref{cor1}}
\ead{croberto@math.cnrs.fr}

\author[2]{B. Zegarlinski}
\ead{bzegarli@math.univ-toulouse.fr}

\cortext[cor1]{Corresponding author}

\address[1]{Universit\'e Paris Nanterre, MODAL'X, UMR CNRS 9023, UPL, FP2M, CNRS FR 2036, 200 avenue de la R\'epublique 92000 Nanterre}
\address[2]{Institut de Math\'ematiques de Toulouse ; CNRS UMR5219,
UPS, F-31062 Toulouse Cedex 9, France}




%
%

\date{24 February 2022}

\begin{abstract}
We investigate in a systematic way hypercontractivity property in Orlicz
spaces for Markov semi-groups related to homogeneous and non homogeneous
diffusions in $\mathbb{R}^{n}$. We provide an explicit construction of
a family of Orlicz functions for which we prove that the associated hypercontractivity
property is equivalent to a suitable functional inequality.
\end{abstract}


%
%


\maketitle

\section{Introduction}
\label{sec1}

The first aim of this paper is to give a unified setting for strong contractivity
properties of Markov semi-group to be satisfied with respect to suitable
family of Luxembourg norms in Orlicz spaces.

Initiated by Nelson in the late sixties
\cite{nelson1,nelson2,nelson3} in quantum field theory, the notion of hypercontractivity
of the Ornstein-Ulhenbeck process was put in light by Gross' seminal work
\cite{gross}. One of the main observation of Gross is that hypercontractivity
is equivalent to the so called log-Sobolev inequality. See also
\cite{federbush,stam} for earlier papers on related topic.

More precisely, let $\gamma _{n}$ be the standard Gaussian measure on
$\mathbb{R}^{n}$. Then the Ornstein-Uhlenbeck semi-group
$(P_{t})_{t \geq 0}$, whose infinitesimal generator is
$L:=\Delta - x \cdot \nabla $ (with the dot sign standing for the Euclidean
scalar product), is reversible with respect to $\gamma _{n}$ and satisfies
the following remarkable \emph{hypercontractivity property}: for any
$f \colon \mathbb{R}^{n} \to \mathbb{R}$ smooth enough it holds
\begin{equation*}
\| P_{t}f \|_{q(t)} \leq \|P_{s} f\|_{q(s)}, \qquad s \leq t
\end{equation*}
where $q(t)=1+(q(0)-1)e^{2t}$, $q(0) \geq 1$, and
$\|g \|_{p}^{p} := \int |g|^{p} d\gamma _{n}$, $p \geq 1$. Such a contraction
property is equivalent \cite{gross} to the following log-Sobolev inequality:
for any $f \colon \mathbb{R}^{n} \to \mathbb{R}$ smooth enough, it holds
\begin{equation*}
\mathrm{Ent}_{\gamma _{n}}(f^{2}):=\int f^{2} \log f^{2} d\gamma _{n}
- \int f^{2} d\gamma _{n} \log \int f^{2} d\gamma _{n} \leq 2 \int |
\nabla f|^{2} d\gamma _{n} .
\end{equation*}
Using Gross' paper and $\Gamma ^{2}$-calculus of Bakry-Emery
\cite{bakry-emery,bakry}, it can be immediately proved that any semi-group
associated to a diffusion of the form
$L:=\Delta - \nabla V \cdot \nabla $, with $V$ satisfying
$\mathrm{Hess}(V) \geq \rho >0$, as a matrix, enjoys the hypercontractivity
property as above with reference measure having density $e^{-V}$ with respect
to the Lebesgue measure and $q(t)=1+(q(0)-1)e^{(4/\rho ) t}$,
$q(0) \geq 1$.

From the seventies, both the hypercontractivity property and the log-Sobolev
inequality found a huge amount of applications in various fields,
including Analysis (isoperimetry, concentration of measure phenomenon,
convex geometry), Statistical mechanics, Information Theory and others.
Giving an exhaustive presentation of the literature is out of reach. We
refer to the textbooks \cite{ane,ledouxbook,GZ,BGL,massart,odonnell} for
an introduction and references.

Now, let $\Phi \colon \mathbb{R}_{+} \to \mathbb{R}_{+}$ be a continuous
convex function satisfying $\Phi (x)=0$ iff $x=0$. Later on we may call
such a function a \emph{Young function}.\footnote{Note however that, usually,
one does not require in the definition of a Young function neither the
regularity assumption, nor the condition $\phi (x)=0$ iff $x=0$.} Then,
given $f\colon \mathbb{R}^{n} \to \mathbb{R}$ such that
$\int \Phi (\alpha f)d\gamma _{n} < + \infty $ for some $\alpha >0$, one
can define the so-called Luxembourg norm associated to $\Phi $ and
$\gamma _{n}$ as
\begin{equation*}
\| f \|_{\Phi } = \inf \left \{  \lambda >0 : \int \Phi \left (
\frac{|f|}{\lambda } \right ) d\gamma _{n} {\leq 1} \right
\}  .
\end{equation*}
The power function $\Phi (x)=|x|^{p}$, $p \geq 1$, trivially corresponds
to the usual $\mathbb{L}_{p}$-norm introduced above
$\|f\|_{\Phi }= \| f\|_{p}$. The space of all functions with finite Luxembourg
norm will be denoted by $\mathbb{L}_{\Phi }(\gamma _{n})$ (or simply
$\mathbb{L}_{\Phi }$ when there is no confusion, note however that norms
are always computed with an underlying measure).

With this definition at hand, for the family of Young functions
$\Phi _{t}(x)=|x|^{q(t)}$, $x \in \mathbb{R}$, $t \geq 0$ with
$q(t)=1+(q(0)-1)e^{2t}$, $q(0) \geq 1$, the hypercontractivity above can
be restated as follows
\begin{equation*}
\| P_{t} f \|_{\Phi _{t}} \leq \| P_{s} f\|_{\Phi _{s}}, \qquad s
\leq t .
\end{equation*}
In other words, the Ornstein-Uhlenbeck semi-group is a contraction along
the family of Orlicz spaces $(\mathbb{L}_{\Phi _{t}})_{t \geq 0}$.

Following Gross' ideas, in \cite{barthe-cattiaux-roberto} the authors proved
that some contraction property along a different type of family of Orlicz
spaces could hold. Consider the following infinitesimal generator\footnote{More
precisely one should consider a regularized version of $|x|^{\alpha }$ in
a neighborhood of the origin. For the sake of simplicity we
may avoid such technical considerations in this introduction, that are
irrelevant for our purpose, and we refer to
\cite{barthe-cattiaux-roberto} for details.} in dimension $n$,
$L:=\Delta - \nabla V \cdot \nabla $, with
$V(x) = \sum _{i=1}^{n} |x_{i}|^{\alpha }$, $\alpha \in [1,2]$,
$x \in \mathbb{R}^{n}$. Denote by $(P_{t})_{t \geq 0}$ the associated semi-group
and by $\mu (dx)=Z^{-1}e^{-V(x)}dx$, $x \in \mathbb{R}^{n}$, the associated
reversible probability measure, $Z:=\int e^{-V(x)}dx$ being the normalization
constant. Finally define $\Phi _{t}(x)=|x|^{p}e^{q(t)F(x)}$ with
$F(x):=\log (1+x)^{2(\alpha -1)/\alpha } - \log (2)^{2(\alpha -1)/
\alpha }$, $q(t)=Ct$ for some constant $C >0$, and $p >1$. By construction
$\mathbb{L}_{\Phi _{t}} \subset \mathbb{L}_{\Phi _{s}} \subset
\mathbb{L}_{p}$ for any $s \leq t$ and
$\mathbb{L}_{\Phi _{t}} \not \subset \mathbb{L}_{p+\varepsilon }$ for any
$\varepsilon >0$, $t \geq 0$ and $\alpha \in [1,2)$ (since
$e^{q(t)F(x)} \ll |x|^{\varepsilon }$ near infinity, for any
$\alpha \in [1,2)$).

In \cite{barthe-cattiaux-roberto} it is proved that
$(P_{t})_{t \geq 0}$ is a contraction along the family of Orlicz spaces
$(\mathbb{L}_{\Phi _{t}})_{t \geq 0}$: namely that, for any
$s \leq t$, it holds
$\| P_{t} f \|_{\Phi _{t}} \leq \| P_{s} f\|_{\Phi _{s}}$. Moreover, such
a contraction property is equivalent to the following, known as $F$-Sobolev
inequality (\cite{Rosen,W00}): for any
$f \colon \mathbb{R}^{n} \to \mathbb{R}$ it holds
\begin{equation*}
\int f^{2} F \left ( \frac{f^{2}}{\int f^{2} d\gamma _{n}} \right ) d
\mu \leq C' \int |\nabla f|^{2} d\mu
\end{equation*}
where $C'$ is a constant that depends on $C$ and $\alpha $. Note that
$\alpha =2$ corresponds to the Gaussian case depicted above. Such inequalities
and contraction properties were used to establish dimension free isoperimetric
inequalities and concentration properties for $\mu $
\cite{barthe-cattiaux-roberto,BCR07}. We refer the reader to
\cite{W14} for explicit criterion for a $F$-Sobolev inequality to hold,
and to \cite{W00} for associated contraction property of the semi-group.

Motivated by the previous two fundamental examples, the aim of this paper
is to investigate on contraction properties
$\| P_{t} f \|_{\Phi _{t}} \leq \| P_{s} f\|_{\Phi _{s}}$,
$s \leq t$, along abstract general family of Orlicz spaces
$(\mathbb{L}_{\Phi _{t}})_{t \geq 0}$, together with possible connection
with functional inequalities of $F$-Sobolev type.

The second objective of the paper is to explore a more general setting
which would include inhomogeneous diffusion operators associated to one
parameter families of probability measures that we now introduce. Consider
$L_{t}:=\Delta - \nabla V_{t} \cdot \nabla $, $t \geq 0$, on
$\mathbb{R}^{n}$, with $V_{t}$ smooth enough and such that
$\int e^{-V_{t}}=1$ so that $\mu _{t}(dx)=e^{-V_{t}(x)}dx$ is a probability
measure on $\mathbb{R}^{n}$ for all $t \geq 0$. The associated semi-group
will be denoted by $(P^{(t)}_{s})_{s \geq 0}$ (we refer to e.g.
\cite{bakry,GZ} for its construction and related technicalities) which
is reversible in $\mathbb{L}_{2}(\mu _{t})$. One wishes to obtain contraction
bounds of the type
$\| P_{t}^{(t)} f \|_{\Phi _{t}} \leq m(t,s) \| P_{s}^{(s)} f\|_{
\Phi _{s}}$, $s \leq t$ for some function $m$ possibly equal to~1.

Thus in the more general setting we not only change with time the Orlicz
functions, but also the underlying probability measures. Here we are interested
in a class of flows through Orlicz spaces and how
it relates to an action of contractions.

Besides interesting \emph{generalizations}, we hope that our results
can be used in the future to study linear and nonlinear parabolic time
dependent problems. Note that in case of a time dependent parabolic problem
of the form
\begin{equation*}
\begin{split} \partial _{t} u & = L u + \beta _{t}\cdot \nabla u
\equiv L_{t} u
\\
u_{|t=0} & =f
\end{split}
\end{equation*}
under suitable conditions on the coefficient $\beta _{t}$, one can hope
to approximate the solution on small intervals
$s\in [t_{n},t_{n+1}]$ by $P_{s-t_{n}}^{(t_{n})}u_{t_{n}}$. Then one needs
to setup a suitable framework to control convergence of such approximation
when $\sup _{n}|t_{n+1}-t_{n}|\to 0$. While we mention here
as an example a linear problem, we remark that nonlinear semigroups with
hypercontractivity properties has been studied in \cite{FRZ} and one could
possibly extend the above given idea to the nonlinear time dependent parabolic
problems.

Moreover, as suggested to us by a referee, since Gross's theorem
is established for symmetric Markov processes associated with Dirichlet
forms, it is reasonable to conjecture that most of the result of this paper
can be extended to such an abstract framework and leave this to future
investigation.

After Section~\ref{sec:technical}, that collects some technical facts about
Orlicz functions-norms, we deal in Section~\ref{sec:homogeneous} with the
homogeneous setting.

Our first main theorem is Theorem~\ref{th:hyp} that asserts that, for any
properly chosen family of Orlicz spaces (see Section~\ref{sec:standard}), we have
\begin{equation*}
\| P_{t} f \|_{\Phi _{t}} \leq \| P_{s} f\|_{\Phi _{s}}
\end{equation*}
if and only if some inequality of log-Sobolev type holds. Theorem~\ref{th:hyp} encompasses the above two known fundamental examples and can
therefore be seen as a generalization of Gross' theorem.

Section~\ref{sec:inhomogeneous} is devoted to the time-inhomogeneous setting.
Our second main result is {\ref{th:main1} which constitutes some
analog of Gross' theorem for inhomogeneous Markov semi-groups.

\medskip

\noindent
\textbf{Acknowledgment.} We warmly thank T. Tao for useful discussion
on the topic of this paper, and the anonymous referees for their suggestions
to improve its content and presentation.

\section{Technical preparations}
\label{sec:technical}

In this section we collect some useful technical facts on various aspects
of Young Functions and Luxembourg norms.

\subsection{Youngs functions}
\label{sec:Young}

An even continuous convex function
$\Phi \colon \mathbb{R} \to \mathbb{R}_{+}$ satisfying $\Phi (x)=0$ iff
$x=0$ is called a \emph{Young function}. If in addition
$\lim _{x \to 0} \Phi (x)/x=0$,
$\lim _{x \to \infty } \Phi (x)/x = + \infty $, and
$\Phi (\mathbb{R}) \subset \mathbb{R}^{+}$, $\Phi $ is called a
\emph{nice Young function}, or \emph{$N$-function} \cite{rao-ren}.

Classical examples include $\Phi (x)=|x|^{p}$, $p \geq 1$ which is a nice
Young function only for $p>1$; $\Phi (x)=e^{|x|}-|x|-1$,
$\Phi (x)=e^{|x|^{\delta }}-1$, $\delta >1$ are nice Young functions.

We say that $\Phi $ satisfies the $\Delta _{2}$-condition if for some
$K >0$ and all $x \geq 0$, it holds $\Phi (2x) \leq K\Phi (x)$. A useful
consequence of the $\Delta _{2}$-condition is the fact that
$x\Phi '(x)$ compares to $\Phi $. More precisely,
%
\begin{equation}
\label{eq:delta2}
\Phi (x) \leq x \Phi '(x) \leq (K-1) \Phi (x) .
\end{equation}
The first inequality follows from the convexity property of $\Phi $ and
$\Phi (0)=0$, while the second is a consequence of the $\Delta _{2}$-condition
and $\Phi (2x)-\Phi (x)=\int _{x}^{2x}\Phi '(t)dt \geq \Phi '(x)x$.

Given a Young function $\Phi $ and a probability measure $\mu $, for any
$f \colon \mathbb{R}^{n} \to \mathbb{R}$ we set
\begin{equation*}
\| f \|_{\Phi } := \inf \left \{  \lambda >0: \int \Phi \left (
\frac{|f|}{\lambda } \right )d\mu {\leq 1} \right \}  \in [0,
\infty ]
\end{equation*}
with the convention that $\inf \emptyset = +\infty $. When useful we may
write $\| f \|_{\Phi ,\mu }$ to emphasize the underlying measure.

\subsection{Derivative of Luxembourg norm}
\label{sec2.2}

Here we give an explicit expression of the derivative with respect to time
of the following function $t \mapsto \| P_{t}f\|_{\Phi _{t}}$ which will
constitute the starting point of our investigations.

In the sequel we will use the following notations. Given a family of twice differentiable Young functions
$(\Phi _{t})_{t \geq 0}=(\Phi (t,x))_{t \geq 0}$, we denote by
$\dot{\Phi }_{t}$ the derivative with respect to $t$, and by
$\Phi _{t}'$ and $\Phi _{t}''$ the first and second order derivative with
respect to the second variable $x$.

Consider the inhomogeneous diffusion generator
$L_{t}:=\Delta - \nabla V_{t} \cdot \nabla $, $t \geq 0$, on
$\mathbb{R}^{n}$, with $V_{t}$ sufficiently smooth and such that
$\int e^{-V_{t}}dx=1$ so that $\mu _{t}(dx)=e^{-V_{t}(x)}dx$ is a probability
measure on $\mathbb{R}^{n}$ for all $t \geq 0$. Denote by
$(P^{(t)}_{s})_{s \geq 0}$ the associated semi-group. By construction
$L_{t}$ is symmetric in $\mathbb{L}_{2}(\mu _{t})$ and the following integration
by parts formula holds for any differentiable function
$\Psi \colon \mathbb{R} \to \mathbb{R}$, any
$f,g \colon \mathbb{R}^{n} \to \mathbb{R}$, such that
$\Psi (f), \nabla \Psi (f)\in \mathbb{L}_{2}(\mu _{t})$ and $g$ is in the
domain of $L_{t}$.
%
\begin{equation}
\label{eq:IBP}
\int \Psi (f) L_{t} g d\mu _{t} = - \int \Psi '(f) \nabla f \cdot
\nabla g d\mu _{t} .
\end{equation}
As we explain in Appendix~\ref{appA}, formally we have
\begin{equation*}
\partial _{t} P_{t}^{(t)}f = L_{t} f + {\mathcal{V}}_{t}f
\end{equation*}
 We prove the following differential property.

\begin{lemma}
\label{lem:technical}
Let $f \colon \mathbb{R}^{n} \to \mathbb{R}$ be a smooth bounded
function not equal to zero a.e. and $(\Phi _{t})_{t \geq 0}$ be a family
of $\mathcal{C}^{2}$ Young functions. Let
$N(t):= \| P_{t}^{(t)} f \|_{\Phi _{t}}$ and
$g:=\frac{P_{t}^{(t)}f}{N(t)}$, $t \geq 0$. Suppose
$\nabla V_{t} \cdot \nabla \dot{V}_{t} - \Delta \dot{V}_{t}= - L_{t}
\dot{V}_{t} $ is $\mu _{t}$-integrable. Then, it holds
\begin{align*}
N'(t) \int g \Phi _{t}'(g) d\mu _{t} &= N(t) \left (\int \dot{\Phi }_{t}(g)
d\mu _{t} - \int {\Phi _{t}''}(g) |\nabla g|^{2} d\mu _{t} -
\int \Phi _{t}(g)\dot{V_{t}}d\mu _{t} \right )
\\
& \quad + \int \int _{0}^{t} P_{t-s}^{(t)} f \nabla P_{s}^{(t)}(\Phi _{t}'(g))
\cdot \nabla \dot{V}_{t} dsd\mu _{t}
\\
& \quad - \int \left [ \nabla V_{t} \cdot \nabla \dot{V}_{t} -
\Delta \dot{V}_{t} \right ] \int _{0}^{t} P_{t-s}^{(t)} f P_{s}^{(t)} (
\Phi _{t}'(g)) ds d\mu _{t}
\end{align*}
In particular, when $V_{t}$ does not depend on $t$ (homogeneous case),
the latter reduces to
\begin{equation*}
N'(t) \int g \Phi _{t}'(g) d\mu = N(t) \left (\int \dot{\Phi }_{t}(g) d
\mu - \int \phi _{t}''(g) |\nabla g|^{2} d\mu \right ) .
\end{equation*}
\end{lemma}

\begin{remark}
We assumed $\mathcal{C}^{2}$ for Young functions for simplicity. Most of
the results in this paper can easily be understood for any Young function
using the notion of second order derivative in the sense of Aleksandrov.
\end{remark}

\begin{proof}
Let $f \colon \mathbb{R}^{n} \to \mathbb{R}$ be a smooth bounded
function not equal to zero a.e. From the  definition
of the Luxembourg norm, we observe that for any $t \geq 0$,
$\int \Phi _{t} \left ( \frac{P_{t}^{(t)}f}{N(t)}\right ) d\mu _{t} =1$.
Therefore, taking the derivative, we get
\begin{align*}
\int \dot{\Phi }_{t}(g) d\mu _{t} + \int \Phi _{t}'(g) \frac{d}{dt}
\left ( \frac{P_{t}^{(t)}f}{N(t)} \right ) d\mu _{t} - \int \Phi _{t}(g)
\dot{V_{t}}d\mu _{t} = 0
\end{align*}
where as already mentioned the dot stands for the derivative with respect
to the variable $t$. Observe that,
\begin{equation*}
\frac{d}{dt} \left ( \frac{P_{t}^{(t)}f}{N(t)} \right ) =
\frac{\dot{P}_{t}^{(t)}f}{N(t)} + \frac{L_{t}P_{t}^{(t)}f}{N(t)} -
\frac{P_{t}^{(t)}f N'(t)}{N(t)^{2}} = \frac{\dot{P}_{t}^{(t)}f}{N(t)} +
L_{t} g - \frac{N'(t)}{N(t)}g
\end{equation*}
where we set
\begin{align*}
\dot{P}_{t}^{(t)}f & := \lim _{\varepsilon \to 0}
\frac{P_{t}^{(t+\varepsilon )}f - P_{t}^{(t)}f}{\varepsilon } = \lim _{
\varepsilon \to 0} \frac{1}{\varepsilon }\int _{0}^{t} \frac{d}{ds}
\left ( P_{s}^{(t+\varepsilon )}(P_{t-s}^{(t)} f) \right )ds
\\
& = \lim _{\varepsilon \to 0} \frac{1}{\varepsilon } \int _{0}^{t} P_{s}^{(t+
\varepsilon )} \left ( [L_{t+\varepsilon } - L_{t}] P_{t-s}^{(t)} f
\right ) ds
\\
& = \int _{0}^{t} P_{s}^{(t)} \left ( - \nabla \dot{V}_{t} \cdot
\nabla P_{t-s}^{(t)} f\right ) ds .
\end{align*}
Therefore, using \eqref{eq:IBP}, we get
\begin{align*}
\int \Phi _{t}'(g) \frac{d}{dt} \left ( \frac{P_{t}^{(t)}f}{N(t)}
\right ) d\mu _{t} & = -\frac{1}{N(t)} \int \int _{0}^{t} \Phi _{t}'(g)
P_{s}^{(t)} \left ( \nabla \dot{V}_{t} \cdot \nabla P_{t-s}^{(t)} f
\right ) dsd\mu _{t}
\\
& \quad - \int {\Phi _{t}''}(g) |\nabla g|^{2} d\mu _{t} -
\frac{N'(t)}{N(t)} \int g\Phi _{t}'(g) d\mu _{t} .
\end{align*}
The previous computations lead to
\begin{align*}
\frac{N'(t)}{N(t)} \int g \Phi _{t}'(g) d\mu _{t} &= \int \dot{\Phi }_{t}(g)
d\mu _{t} - \int {\Phi _{t}''}(g) |\nabla g|^{2} d\mu _{t} -
\int \Phi _{t}(g)\dot{V_{t}}d\mu _{t}
\\
& \quad -\frac{1}{N(t)} \int \int _{0}^{t} \Phi _{t}'(g) P_{s}^{(t)}
\left ( \nabla \dot{V}_{t} \cdot \nabla P_{t-s}^{(t)} f\right ) dsd
\mu _{t}
\end{align*}
and we are left with the study of the last term on the right hand side
of the latter. By reversibility of the semi-group, we have
\begin{align*}
\int \Phi _{t}'(g) P_{s}^{(t)} \left ( \nabla \dot{V}_{t} \cdot
\nabla P_{t-s}^{(t)} f\right ) d\mu _{t} & = \int P_{s}^{(t)}(\Phi _{t}'(g))
\nabla \dot{V}_{t} \cdot \nabla P_{t-s}^{(t)} f d\mu _{t}
\\
& = \int P_{t-s}^{(t)} f {\nabla ^{*}_{t}} \left (P_{s}^{(t)}(
\Phi _{t}'(g)) \nabla \dot{V}_{t} \right )d\mu _{t}
\end{align*}
where ${\nabla ^{*}_{t}}$ is the adjoint of $\nabla $ in
$\mathbb{L}_{2}(\mu _{t})$, namely such that
$\int \! u \nabla v d\mu _{t} = \int \! v {\nabla ^{*}_{t}} u d
\mu _{t}$. One can see that
${\nabla ^{*}_{t}}=-{\mathrm{div}} + \nabla V_{t}$
where $\nabla V_{t}$ acts multiplicatively. Therefore,
\begin{align*}
\int \Phi _{t}'(g) P_{s}^{(t)} \left ( \nabla \dot{V}_{t} \cdot
\nabla P_{t-s}^{(t)} f\right ) d\mu _{t} & = - \int P_{t-s}^{(t)} f
\nabla P_{s}^{(t)}(\Phi _{t}'(g)) \cdot \nabla \dot{V}_{t} d\mu _{t}
\\
&  + \int P_{t-s}^{(t)} f P_{s}^{(t)} (\Phi _{t}'(g)) \left [
\nabla V_{t} \cdot \nabla \dot{V}_{t} - \Delta \dot{V}_{t} \right ] d
\mu _{t}
\end{align*}
From this the desired result follows.
\end{proof}

\subsection{Expansion of the square}
\label{sec2.3}

Here we may recall the \emph{expansion of the square} method or, as it is
called in \cite{HZ}, $U$-bounds. That is the bounds obtained by using Leibnitz
rule together with integration by parts as follows. Let
$U \colon \mathbb{R}^{n} \to \mathbb{R}$ be such that
$\int e^{-U} dx < \infty $. Then, for any differentiable function
$f \colon \mathbb{R}^{n} \to \mathbb{R}$, one has
%
\begin{equation}
\label{eq:Ubound}
\int f^{2}( |\nabla U|^{2} - 2\Delta U ) e^{-U} dx \leq 4 \int |
\nabla f|^{2} e^{-U} dx .
\end{equation}
In fact, expanding the square, one has
\begin{align*}
0 & \leq \int |\nabla (fe^{-U/2})|^{2} dx \\
&  = \int |\nabla f|^{2} e^{-U}dx
- \int f \nabla f \cdot \nabla U e^{-U}dx + \frac{1}{4}\int f^{2} |
\nabla U|^{2} e^{-U} dx .
\end{align*}
The expected inequality {\eqref{eq:Ubound}} then follows by applying an integration
by parts on the cross term.

The \emph{expansion of the square} method revealed to be very powerful.
It can be used for instance to prove Hardy's inequality with optimal constant
on $\mathbb{R}^{d}$, $d\geq 3$, or Poincar\'{e} inequality for the Gaussian
measure. We refer the interested reader to \cite{HZ,DV} for more results
and references.

\section{Hypercontractivity for homogeneous Markov semi-groups}
\label{sec:homogeneous}

In this section our aim is to introduce the notion of
\emph{standard Orlicz family} that will play a key role for proving the
equivalence between some functional inequality and a hypercontractivity
property along the corresponding family of Orlicz spaces. We need first
to analyze how to get a hypercontractivity property along
a general family of Orlicz spaces.

All along the section we set $L=\Delta - \nabla V \cdot \nabla $, with
$V$ smooth enough and such that $\mu (dx)=e^{-V}dx$ is a probability measure
on $\mathbb{R}^{n}$. We denote by $(P_{t})_{t \geq 0}$ the associated semi-group
which is reversible with respect to $\mu $. Orlicz spaces and their corresponding
Luxembourg norms are understood with respect to $\mu $.

\subsection{Hypercontractivity along Orlicz spaces}
\label{sec3.1}

Using Lemma~{\ref{lem:technical}} we first prove that hypercontractivity
is a direct and immediate consequence of some family of functional inequalities.
Our second result shows how that family can, under some assumptions, be
reduced to one single functional inequality of log-Sobolev-type.

\begin{proposition}
\label{prop:hyp}
Let $(\Phi _{t})_{t \geq 0}$ be a family of ${\mathcal{C}}^{2}$ Young functions.
Assume that for any $t \geq 0$, any sufficiently smooth function $f$, we
have
%
\begin{equation}
\label{eq:f}
\Vert f \Vert _{\Phi _{t}}^{2} \int \dot{\Phi }_{t} \left (
\frac{f}{\Vert f \Vert _{\Phi _{t}}} \right ) d \mu \leq \int \Phi _{t}''
\left ( \frac{f}{\Vert f \Vert _{\Phi _{t}}} \right ) |\nabla f|^{2} d
\mu .
\end{equation}
Then, for any $t \geq s$,
\begin{equation*}
\Vert P_{t} f \Vert _{\Phi _{t}} \leq \Vert P_{s} f \Vert _{\Phi _{s}}
.
\end{equation*}
\end{proposition}

\begin{proof}
We need to prove is that
$N : t \mapsto \Vert P_{t} f \Vert _{\Phi _{t}}$ is non-increasing. Lemma~{\ref{lem:technical}} asserts that
\begin{align*}
\frac{N'(t)}{N(t)} \int g \Phi _{t}' \left ( g \right ) d\mu = \int
\dot{\Phi }_{t} \left (g\right ) d\mu - \int \Phi _{t}'' \left ( g
\right ) |\nabla g |^{2} d\mu
\end{align*}
where $g:=\frac{P_{t} f}{N(t)}$. Since for any $t \geq 0$,
$\Phi _{t}$ is a Young function, it satisfies
$x \Phi _{t}'(x) \geq 0$ for any $x \in \mathbb{R}$. It follows by {\eqref{eq:f}} that $N'(t) \leq 0$ which is the expected result.
\end{proof}

Using an isometry between $\mathbb{L}_{\Phi _{t}}$ and
$\mathbb{L}_{\Phi _{s}}$, we may translate {\eqref{eq:f}} for
$\Phi _{s}$ into a similar inequality for $\Phi _{t}$, therefore reducing
the family of inequalities {\eqref{eq:f}} possibly to a single one.

\begin{proposition}
\label{prop:move}
Let $(\Phi _{t})_{t \geq 0}$ be a family of ${\mathcal{C}}^{2}$ Young functions.
Assume that for some $t, s \geq 0$ there exist two positive constants
$C(t,s)$ and $\widetilde{C}(t,s)$ such that\\
$(i)$
\begin{equation*}
\dot{\Phi }_{t} (\Phi _{t}^{-1}) \leq C(t,s) \dot{\Phi }_{s} (\Phi _{s}^{-1})
,
\end{equation*}
$(ii)$
\begin{equation*}
\frac{\Phi _{t}''}{{\Phi _{t}'}^{2}} \circ \Phi _{t}^{-1} \geq
\widetilde{C}(t,s) \frac{\Phi _{s}''}{{\Phi _{s}'}^{2}} \circ \Phi _{s}^{-1}
.
\end{equation*}
Assume furthermore that for some constant $c>0$ and for any $f$ (smooth
enough), it holds
%
\begin{equation}
\label{eq:f1}
\Vert f \Vert _{\Phi _{s}}^{2} \int \dot{\Phi }_{s} \left (
\frac{f}{\Vert f \Vert _{\Phi _{s}}} \right ) d \mu \leq c \int \Phi _{s}''
\left ( \frac{f}{\Vert f \Vert _{\Phi _{s}}} \right ) |\nabla f|^{2} d
\mu .
\end{equation}
Then, for any $f$ smooth enough it holds
\begin{equation*}
\Vert f \Vert _{\Phi _{t}}^{2} \int \dot{\Phi }_{t} \left (
\frac{f}{\Vert f \Vert _{\Phi _{t}}} \right ) d \mu \leq c
\frac{C(t,s)}{\widetilde{C}(t,s)} \int \Phi _{t}'' \left (
\frac{f}{\Vert f \Vert _{\Phi _{t}}} \right ) |\nabla f|^{2} d\mu .
\end{equation*}
\end{proposition}

\begin{proof}
Let
\begin{equation*}
\begin{array}{r@{\quad }c@{\quad }l}
I_{s,t} : \mathbb{L}_{\Phi _{t}} & \rightarrow & \mathbb{L}_{\Phi _{s}}
\\
f & \mapsto & \Vert f \Vert _{\Phi _{t}} \Phi _{s}^{-1} \circ \Phi _{t}
\left ( \frac{f}{\Vert f \Vert _{\Phi _{t}}} \right ) .
\end{array}
\end{equation*}
For any $f \in \mathbb{L}_{\Phi _{t}}$, by the very definition of the Luxembourg
norm, it holds
$\Vert I_{s,t}(f)\Vert _{\Phi _{s}}=\Vert f\Vert _{\Phi _{t}}$. Therefore,
$I_{s,t}(f)$ is an isometry between the two Orlicz spaces
$\mathbb{L}_{\Phi _{t}}$ and $\mathbb{L}_{\Phi _{s}}$. Applying {\eqref{eq:f1}} to $I_{s,t}(f)$ leads to
\begin{eqnarray*}[ll]
 \Vert f \Vert _{\Phi _{t}}^{2} \int \dot{\Phi }_{s} \left (\Phi _{s}^{-1}
\circ \Phi _{t} \Big ( \frac{f}{\Vert f \Vert _{\Phi _{t}}} \Big )
\right ) d \mu \leq
\\
 \qquad \qquad \qquad \qquad c \int \Phi _{s}'' \circ \Phi _{s}^{-1}
\circ \Phi _{t} \left ( \frac{f}{\Vert f \Vert _{\Phi _{t}}} \right )
\frac{\Phi _{t}'\left (\frac{f}{\Vert f\Vert _{\Phi _{t}}}\right )^{2}
|\nabla f|^{2} }{\Phi _{s}' \circ \Phi _{s}^{-1}\circ \Phi _{t}
\left (\frac{f}{\Vert f\Vert _{\Phi _{t}}}\right )^{2}} d\mu .
\end{eqnarray*}
The result follows by $(i)$ and $(ii)$.
\end{proof}

The simplest example is given by the $\mathbb{L}_{p}$ scale
$\Phi _{t}(x)=|x|^{q(t)}$ for some function $q$ that we assume to be increasing.
Then, it holds
\begin{equation*}
\dot{\Phi }_{t} (\Phi _{t}^{-1}) = \frac{q'(t)}{q(t)} x \log x \qquad
\text{ and } \qquad \frac{\Phi _{t}''}{{\Phi _{t}'}^{2}} \circ \Phi _{t}^{-1}
= \frac{q(t) - 1}{q(t)} \frac{1}{x}.
\end{equation*}
Therefore assumptions $(i)$ and $(ii)$ hold with
$C(t,s)=\frac{q'(t)q(s)}{q(t)q'(s)}$ and
$\widetilde{C}(t,s)= \frac{(q(t)-1)q(s)}{q(t)(q(s)-1)}$. In particular,
the choice $q(t) = 1 + e^{(4/\rho ) t}$, $\rho >0$, guarantees that
$C(t,s) = \widetilde{C}(t,s)$ for all $s,t$. Hence, the family of inequalities {\eqref{eq:f1}} are all equivalent to {\eqref{eq:f1}} with $s=0$, which reads
\begin{equation*}
\mathrm{Ent}_{\mu }(f^{2}) \leq \rho \int |\nabla f|^{2} d\mu
\end{equation*}
since $\Phi _{0}=|x|^{2}$ (and therefore
$\dot{\Phi }_{0}(x)= (2/\rho ) x^{2} \log x^{2}$ and
$\Phi _{0}''(x)=2$). This is the log-Sobolev inequality and therefore Proposition~{\ref{prop:move}} is just one direction in Gross' theorem \cite{gross}.

In the above example, both $\dot{\Phi }_{t} (\Phi _{t}^{-1})$ and
$\frac{\Phi _{t}''}{{\Phi _{t}'}^{2}} \circ \Phi _{t}^{-1}$ are of the
form $a(t)b(x)$. Based on this simple observation, we may construct a generic
family of Orlicz functions that, by construction, will automatically satisfies
assumptions $(i)$ and $(ii)$ of the latter. This is the object of the next
section.

\subsection{The standard Orlicz family}
\label{sec:standard}

We define a large class of family of $N$-functions that we will call the
\emph{standard Orlicz family}.

\begin{definition}[standard Orlicz family]
\label{def:sof}
Let $F : (0,\infty )\to \mathbb{R}$ be a $\mathcal{C}^{2}$ increasing
function with $F(1)=0$. Assume that
$(0,\infty ) \ni x \mapsto xF(x)$ is convex and that $1/xF(x)$ is not integrable
at $x=0$, $x=1$ and $x=+\infty $. Let
$\mathcal{F}_{1} : (0,1) \to \mathbb{R}$ and
$\mathcal{F}_{2} : (1,+\infty ) \to \mathbb{R}$ be two primitives of
$x \mapsto 1/(xF(x))$.

Let $\Phi _{0}$ be a nice Young function and $x_{o}$ the unique positive
point such that $\Phi _{0}(x_{o})=1$. We assume that
$- \left (\frac{\Phi _{0}}{\Phi _{0}'} \right )' F(\Phi _{0}) - \Phi _{0}
F'(\Phi _{0})$ is non-increasing on $\mathbb{R}^{+}$ and that
$\Phi _{0}$ is of class $\mathcal{C}^{2}$ on $(0,\infty )$.

Given an increasing function
$\lambda \colon [0,\infty ) \to [0,\infty )$, with $\lambda (0)=0$, we
define the family of functions $(\Phi _{t})_{t \geq 0}$ by
\begin{equation*}
\Phi _{t}(x) =
\begin{cases}
0 & \text{ for } x=0
\\
\mathcal{F}_{1}^{-1} \left ( \mathcal{F}_{1}(\Phi _{0}(x)) + \lambda (t)
\right ) & \text{ for } x \in (0,x_{o})
\\
1 & \text{ for } x = x_{o}
\\
\mathcal{F}_{2}^{-1} \left ( \mathcal{F}_{2}(\Phi _{0}(x)) + \lambda (t)
\right ) & \text{ for } x \in (x_{o},+\infty ) .
\end{cases}
\qquad \forall t>0 .
\end{equation*}

We shall call the family $(\Phi _{t})_{t \geq 0}$ the
\emph{standard Orlicz family built from $F$, $\Phi _{0}$ and $\lambda $}.
\end{definition}

\begin{remark}
The Lemma below will prove that all $\Phi _{t}$ are indeed Young functions
and in fact nice Young functions. This justifies the terminology ``Orlicz
family''. Also, it is not difficult to check that the definition above does
not depend on the choice of the primitives: any two different primitives
lead to the same final function $\Phi _{t}$.
\end{remark}

\begin{example}
\label{ex:sof}
As an example consider $F(x)=\log (x)$ and any nice Young function
$\Phi _{0}$. Then, $\mathcal{F}_{1}(x)=\log (\log (1/x))$,
$x \in (0,1)$ and $\mathcal{F}_{2}(x)=\log (\log (x))$, $x >1$ so that
$\mathcal{F}_{1}^{-1}(x)=e^{-e^{x}}$ and
$\mathcal{F}_{2}^{-1}(x)=e^{e^{x}}$, $x \in \mathbb{R}$. Hence,
$\Phi _{t}(x)=\Phi _{0}^{e^{\lambda (t)}}$. This corresponds to an
$\mathbb{L}_{p}$ scale when $\Phi _{0}(x)=|x|^{q}$ for some $q > 1$. More
specifically, if $q(t) = 1 + e^{(4/\rho ) t}$ and
$\lambda (t)= \log (q(t)/2)$, with $\Phi _{0}(x)=x^{2}$, we have
$\Phi _{t}(x)=|x|^{q(t)}$ and we are back to Gross' setting.

The more general choices $F(x)=\log (1+x)^{\beta }- \log (2)^{\beta }$,
$\beta \in (0,1)$, can also be considered, but lead to non explicit
$\mathcal{F}_{1}$ and $\mathcal{F}_{2}$. Although one can easily give an
asymptotic of the corresponding ${\Phi }_{t}(x)$, when $x$ tends
to $0$ or $+\infty $. For instance, $\Phi _{t}$ is equivalent to
$\Phi _{0} e^{a_{\beta }\lambda (\log \phi _{0})^{\beta }}$ when $x$ tends to
infinity, where $a_{\beta }$ is a numerical constant that depends only on
$\beta $. This amounts to the family of Young functions
$x^{2} e^{ct F(x)}$ considered in
\cite[Section 7]{barthe-cattiaux-roberto}.
\end{example}

In the next lemma we collect some property of the standard Orlicz families.

\begin{lemma}
\label{lem:sof}
Let $F$, $\Phi _{0}$ and $\lambda $ satisfying the assumptions of Definition~{\ref{def:sof}} and let $(\Phi _{t})_{t \geq 0}$ be the standard Orlicz family
built from $F$, $\Phi _{0}$ and $\lambda $. Then,%
\xmlpar $(i)$ $\mathcal{F}_{1}$ and $\mathcal{F}_{2}$ are $\mathcal{C}^{2}$ functions
respectively on $(0,1)$ and $(1,+\infty )$. $\mathcal{F}_{1}$ is decreasing
with
$\lim _{x \to 0} \mathcal{F}_{1}(x) =-\lim _{x \to 1} \mathcal{F}_{1}(x)
= +\infty $. While $\mathcal{F}_{2}$ is increasing with
$\lim _{x \to 1} \mathcal{F}_{2}(x) =-\lim _{x \to +\infty }
\mathcal{F}_{2} (x) = -\infty $. In particular $\Phi _{t}$ is well defined
and continuous. Moreover, for $t \geq s$, $\Phi _{t} \leq \Phi _{s}$ on
$(0,x_{0})$ and $\Phi _{t} \geq \Phi _{s}$ on $(x_{0},+\infty )$.%
\xmlpar $(ii)$ For any $t \geq 0$, $\Phi _{t}$ is a nice Young function of class
$\mathcal{C}^{2}$ on $(0,\infty )$ (with
$\Phi _{t}'(x_{o})=\Phi _{0}'(x_{o})$ and
$\Phi _{t}''(x_{o})=\Phi _{0}''(x_{o})$).
\xmlpar $(iii)$ For any $t \geq 0$,
$\dot{\Phi }_{t} \circ \Phi _{t}^{-1} = \lambda '(t) x F(x)$.%
\xmlpar $(iv)$ For any $t \geq s \geq 0$.
$\frac{\Phi _{t}''}{{\Phi _{t}'}^{2}}\circ \Phi _{t}^{-1} \geq
\frac{\Phi _{s}''}{{\Phi _{s}'}^{2}}\circ \Phi _{s}^{-1}$.%
\xmlpar $(v)$ Assume that $\lambda $ tends to infinity at infinity. Then for any
$f \in \mathbb{L}_{\infty }$, it holds
$\lim _{t \to +\infty }\Vert f \Vert _{\Phi _{t}} = \frac{1}{x_{0}}
\Vert f \Vert _{\infty }$.
\end{lemma}

\begin{proof}
Points $(i)$ and $(iii)$ are simple consequences of the definitions of
the object involved.

It is not difficult but tedious to prove that for all $t \geq 0$,
$\Phi _{t}$ is $\mathcal{C}^{2}$ (we omit the proof). Using that
$\Phi _{t} \leq \Phi _{0}$ on $(0,x_{o})$ and since $\Phi _{0}$ is a nice
Young function we deduce that
$\lim _{x \to 0} \frac{\Phi _{t}(x)}{x}=0$. Similarly,
$\lim _{x \to \infty } \frac{\Phi _{t}(x)}{x}=+\infty $. In order to prove
that $\Phi _{t}$ is a nice Young function it therefore remains to show
that $\Phi _{t}$ is convex. For $x \neq x_{0}$, a simple differentiation
gives
%
\begin{align}
\label{eq:second}
\frac{\Phi _{t}''}{{\Phi _{t}'}^{2}} & = \mathcal{F}'(\Phi _{t})
\left ( -
\frac{\mathcal{F}''(\Phi _{t})}{\mathcal{F}'(\Phi _{t})^{2}} +
\frac{\mathcal{F}''(\Phi _{0})}{\mathcal{F}'(\Phi _{0})^{2}} +
\frac{\Phi _{0}''}{\mathcal{F}'(\Phi _{0}) {\Phi _{0}'}^{2}} \right )
\\
& = \mathcal{F}'(\Phi _{t}) \left ( \left ( \frac{1}{\mathcal{F}'}
\right )'(\Phi _{t}) - \left ( \frac{1}{\mathcal{F}'} \right )'(\Phi _{0})
+ \frac{\Phi _{0}''}{\mathcal{F}'(\Phi _{0}) {\Phi _{0}'}^{2}}
\right )
\nonumber
\end{align}
where $\mathcal{F}=\mathcal{F}_{1}$ when $x \in (0, x_{0})$ and
$\mathcal{F}=\mathcal{F}_{2}$ when $x > x_{0}$. A Taylor expansion of
$\left ( \frac{1}{\mathcal{F}'} \right )'$ at the first order insures that
\begin{equation*}
\frac{\Phi _{t}''}{{\Phi _{t}'}^{2}} = \mathcal{F}'(\Phi _{t})\left ( (
\Phi _{t} - \Phi _{0})\left ( \frac{1}{\mathcal{F}'} \right )'' (
\theta ) +
\frac{\Phi _{0}''}{\mathcal{F}'(\Phi _{0}) {\Phi _{0}'}^{2}} \right )
\end{equation*}
for some $\theta \in (\Phi _{t},\Phi _{0})$ when $x \in (0, x_{0})$ and
$\theta \in (\Phi _{0},\Phi _{t})$ when $x > x_{0}$. Since
$x \mapsto xF(x)$ is convex, $\frac{1}{\mathcal{F}'}$ is convex. It follows
that $\left ( \frac{1}{\mathcal{F}'} \right )'' (\theta ) \geq 0$ and thus
that
$(\Phi _{t} - \Phi _{0}) \left (\frac{1}{F}' \right )'' (\theta ) +
\frac{\Phi _{0}''}{\mathcal{F}'(\Phi _{0}) {\Phi _{0}'}^{2}}$ has the same
sign as $\mathcal{F}'(\Phi _{t})$. This proves that $\Phi _{t}$ is convex.

Next, we deal with Point $(iv)$. From {\eqref{eq:second}} we have with the
same notation as before
\begin{equation*}
\frac{\Phi _{t}''}{{\Phi _{t}'}^{2}} \circ \Phi _{t}^{-1}(x) = -
\frac{\mathcal{F}''(x)}{\mathcal{F}'(x)} + \mathcal{F}'(x) \left (
\frac{\mathcal{F}''(\Phi _{0})}{\mathcal{F}'(\Phi _{0})^{2}} +
\frac{\Phi _{0}''}{\mathcal{F}'(\Phi _{0}) {\Phi _{0}'}^{2}} \right )
\circ \Phi _{t}^{-1}(x) .
\end{equation*}
Note that by hypothesis,
\begin{equation*}
\frac{\mathcal{F}''(\Phi _{0})}{\mathcal{F}'(\Phi _{0})^{2}} +
\frac{\Phi _{0}''}{\mathcal{F}'(\Phi _{0}) {\Phi _{0}'}^{2}} = -
\left (\frac{\Phi _{0}}{\Phi _{0}'} \right )' F(\Phi _{0}) - \Phi _{0}
F'(\Phi _{0})
\end{equation*}
is non-increasing. Thus, by Point $(ii)$ and using the sign of
$\mathcal{F}'(x)$ on each domain $(0,x_{0})$ and $(x_{0},+\infty )$, we
have
\begin{eqnarray*}
\frac{\Phi _{t}''}{{\Phi _{t}'}^{2}} \circ \Phi _{t}^{-1}(x) & \geq & -
\frac{\mathcal{F}''(x)}{\mathcal{F}'(x)} + \mathcal{F}'(x) \left (
\frac{\mathcal{F}''(\Phi _{0})}{\mathcal{F}'(\Phi _{0})^{2}} +
\frac{\Phi _{0}''}{\mathcal{F}'(\Phi _{0}) {\Phi _{0}'}^{2}} \right )
\circ \Phi _{s}^{-1}(x)
\\
& = & \frac{\Phi _{s}''}{{\Phi _{s}'}^{2}} \circ \Phi _{s}^{-1}(x)
\end{eqnarray*}
which is the expected result.

Finally we will prove Point $(v)$. Let $f \in \mathbb{L}_{\infty }$. Then,
$\int \Phi _{t} \Big ( \frac{x_{0} |f|}{\Vert f \Vert _{\infty }} \Big ) d
\mu \leq \Phi _{t} (x_{0})=1$. Hence, by definition of the norm,
$\Vert f \Vert _{\Phi _{t}} \leq \frac{1}{x_{0}} \Vert f \Vert _{\infty }$. In order to prove the bound from below, fix
$\varepsilon >0$ small enough. Then note that for any $x>x_{0}$,
$\lim _{t \to +\infty } \Phi _{t}(x) = +\infty $. Thus
\begin{align*}
\int \Phi _{t} \left (
\frac{|f| x_{0}}{\Vert f \Vert _{\infty }(1-\varepsilon )} \right ) d
\mu & \geq \int _{\{|f| \geq \Vert f \Vert _{\infty }(1-
\frac{\varepsilon }{2}) \}} \Phi _{t} \left (
\frac{|f| x_{0}}{\Vert f \Vert _{\infty }(1-\varepsilon )} \right ) d
\mu
\\
& \geq \Phi _{t} \left ( x_{0}\Big (1+
\frac{\varepsilon }{2(1-\varepsilon )}\Big ) \right ) \mu \left ( \{|f|
\geq \Vert f \Vert _{\infty }(1-\frac{\varepsilon }{2}) \} \right ) \\
& \geq 1
\end{align*}
provided $t$ is large enough. It follows that
$\frac{1}{x_{0}}\Vert f \Vert _{\infty }(1-\varepsilon ) \leq \Vert f
\Vert _{\Phi _{t}}$ for $t$ large enough. This leads to the expected result
and achieves the proof of the lemma.
\end{proof}

\begin{remark}
When $\lambda (t)=\alpha t$ for some $\alpha >0$, the standard Orlicz family
enjoy a shift type property. Indeed, in that case
$\Phi _{t} = \mathcal{F}_{i}^{-1}(\mathcal{F}_{i}(\Phi _{s}) +
\lambda (t-s))$, $i=1, 2$. Therefore, the standard Orlicz families
$(\Phi _{t})_{t \geq 0}$ built from $\Phi _{0}$, $F$ and $\lambda $ and
$(\Psi _{t})_{t \geq 0}$ built from $\Phi _{s}$, $F$ and $\lambda $, satisfy
$\Psi _{t} = \Phi _{t+s}$ for any $t \geq 0$.
\end{remark}

\begin{remark}
\label{rem:x2}
When $\Phi _{0}(x)=x^{2}$,
\begin{equation*}
- \left (\frac{\Phi _{0}}{\Phi _{0}'} \right )' F(\Phi _{0}) - \Phi _{0}
F'(\Phi _{0}) = - \frac{1}{2} F(x^{2}) - x^{2} F'(x^{2}) .
\end{equation*}
Thus, this function is non-increasing if and only if
$\frac{3}{2} F'(x)+ xF''(x) \geq 0$ if and only if
$x \mapsto xF(x^{2})$ is convex. Thus, in that case, one can only assume
that $x \mapsto xF(x^{2})$ is convex (which implies that
$x \mapsto xF(x)$ is convex).
\end{remark}

\subsection{Gross-Orlicz' theorem}
\label{sec3.3}

Thanks to the above definition of the standard Orlicz family, we can state
one of our main results which generalizes Gross's theorem.

\begin{theorem}[Gross-Orlicz]
\label{th:hyp}
Let $(\Phi _{t})_{t \geq 0}$ be a standard Orlicz family built from
$F$, $\Phi _{0}$ and $\lambda $ satisfying the hypotheses of Lemma~{\ref{lem:sof}}. Let $c>0$. Then the following are equivalent%
\xmlpar $(i)$
%
\begin{equation}
\label{eq:FsobGross}
\Vert f \Vert _{\Phi _{0}}^{2} \int \Phi _{0} \Big (
\frac{f}{\Vert f \Vert _{\Phi _{0}}} \Big ) F \left ( \Phi _{0} \Big (
\frac{f}{\Vert f \Vert _{\Phi _{0}}} \Big ) \right ) d \mu \leq c
\int \Phi _{0}'' \left ( \frac{f}{\Vert f \Vert _{\Phi _{0}}} \right )
|\nabla f|^{2} d\mu ;
\end{equation}
for any function $f$ for which the right hand side is well defined;%
\xmlpar $(ii)$ $t \geq s \geq 0$, it holds
\begin{equation*}
\Vert P_{t} f \Vert _{\Phi _{t}} \leq \Vert P_{s} f \Vert _{\Phi _{s}},
\end{equation*}
for any function $f$ for which the right hand side is well defined.%

Moreover $(i) \Rightarrow (ii)$ with any (increasing) $\lambda $ such that
$\frac{\Phi _{t}''}{{\Phi _{t}'}^{2}} \circ \Phi _{t}^{-1} \geq c
\lambda '(t) \frac{\Phi _{0}''}{{\Phi _{0}'}^{2}}\circ \Phi _{0}^{-1}$
for any $t \geq 0$ (in particular, any $\lambda $ satisfying
$\lambda '(t) \leq 1/c$ would do); and $(ii) \Rightarrow (i)$ with
$c= 1/\lambda '(0)$.
\end{theorem}

\begin{proof}
We first prove that $(i)$ implies $(ii)$. Item $(iv)$ of Lemma~{\ref{lem:sof}} guarantees that
$\frac{\Phi _{t}''}{{\Phi _{t}'}^{2}} \circ \Phi _{t}^{-1} \geq c
\lambda '(t) \frac{\Phi _{0}''}{{\Phi _{0}'}^{2}}\circ \Phi _{0}^{-1}$
with $\lambda (t)=t/c$. Hence, the set of functions $\lambda $, increasing,
satisfying
$\frac{\Phi _{t}''}{{\Phi _{t}'}^{2}} \circ \Phi _{t}^{-1} \geq c
\lambda '(t) \frac{\Phi _{0}''}{{\Phi _{0}'}^{2}}\circ \Phi _{0}^{-1}$
for any $t \geq 0$ is non empty and we may fix one of them.

Consider the standard Orlicz family $(\Phi _{t})_{t \geq 0}$ built from
$F$, $\Phi _{0}$ and $\lambda $.

Note that by definition of the standard Orlicz family,
$\dot{\Phi }_{0}=\lambda '(0) \Phi _{0} F(\Phi _{0})$. Thus Inequality {\eqref{eq:FsobGross}} reads as
\begin{equation*}
\Vert f \Vert _{\Phi _{0}}^{2} \int \dot{\Phi }_{0} \left (
\frac{f}{\Vert f \Vert _{\Phi _{0}}} \right ) d\mu \leq \lambda '(0) c
\int \Phi _{0}'' \left ( \frac{f}{\Vert f \Vert _{\Phi _{0}}} \right )
|\nabla f|^{2} d\mu .
\end{equation*}
From the properties proved in Lemma~{\ref{lem:sof}} we can apply Proposition~{\ref{prop:move}} with $C(t,0)=\frac{\lambda '(t)}{\lambda '(0)}$ and
$\widetilde{C}(t,0)=c \lambda '(t)$. We get that, for any $t \geq 0$, any
smooth function $f$ satisfies
\begin{align*}
\Vert f \Vert _{\Phi _{t}}^{2} \int\dot{\Phi }_{t} \left (
\frac{f}{\Vert f \Vert _{\Phi _{t}}} \right ) d \mu& \leq \lambda '(0)c
\frac{C(t,0)}{\widetilde{C}(t,0)} \int \Phi _{t}'' \left (
\frac{f}{\Vert f \Vert _{\Phi _{t}}} \right ) |\nabla f|^{2} d\mu \\
&=
\int \Phi _{t}'' \left ( \frac{f}{\Vert f \Vert _{\Phi _{t}}}
\right ) |\nabla f|^{2} d\mu .
\end{align*}
The result of Point $(ii)$ follows by Proposition~{\ref{prop:hyp}}.

Now we prove that $(ii)$ implies $(i)$. Let
$N(t)=\Vert P_{t} f \Vert _{\Phi _{t}}$. By Lemma~{\ref{lem:technical}} at
$t=0$, we infer that
\begin{align*}
\frac{N'(0)}{N(0)} \int \frac{f}{\Vert f \Vert _{\Phi _{0}}} \Phi _{0}'
\left (\frac{f}{\Vert f \Vert _{\Phi _{0}}} \right ) d \mu
&= \int
\dot{\Phi }_{0} \left (\frac{f}{\Vert f \Vert _{\Phi _{0}}} \right ) d
\mu \\
&\quad {}- \int \Phi _{0}'' \left ( \frac{f}{\Vert f \Vert _{\Phi _{0}}}
\right ) |\nabla \frac{f}{\Vert f \Vert _{\Phi _{0}}} |^{2} d\mu .
\end{align*}
The hypercontractivity property of Point $(ii)$ insures that
$N'(0) \leq 0$. Thus, since $x \Phi _{0}'(x) \geq 0$, we get that
\begin{equation*}
\Vert f \Vert _{\Phi _{0}}^{2} \int \dot{\Phi }_{0} \left (
\frac{f}{\Vert f \Vert _{\Phi _{0}}} \right ) d\mu \leq \int \Phi _{0}''
\left ( \frac{f}{\Vert f \Vert _{\Phi _{0}}} \right ) |\nabla f |^{2} d
\mu .
\end{equation*}
The result follows by the formula
$\dot{\Phi }_{0}=\lambda '(0) \Phi _{0} F(\Phi _{0})$ proved in Lemma~{\ref{lem:sof}}.
\end{proof}

\begin{remark}
\label{rem:go}
When $\Phi _{0}(x)=x^{2}$, Inequality {\eqref{eq:FsobGross}} reads as
\begin{equation*}
\int f^{2} F \left ( \frac{f^{2}}{\mu (f^{2})} \right ) d\mu \leq 2c
\int |\nabla f|^{2} d\mu .
\end{equation*}
This is the usual $F$-Sobolev inequality introduced by Rosen
\cite{Rosen} (see also Wang \cite{W00}), and corresponds to the log-Sobolev
inequality when $F(x)=\log x$.

When $F = \log $, one can consider
$\lambda (t)=\log (1+e^{(4/\rho ) t}) - \log 2$, with $\rho =2c$. Then,
as already mentioned in Example~{\ref{ex:sof}}, the standard Orlicz family
built from $\Phi _{0}(x)=x^{2}$, $F$ and $\lambda $, is
$\Phi _{t}(x)=|x|^{q(t)}$ with $q(t)=1+e^{(4/\rho ) t}$. In that case {\ref{th:hyp}} is nothing but Gross' equivalence between the log-Sobolev
inequality and the hypercontractivity in $\mathbb{L}_{p}$ scale recalled
in the introduction.

{\ref{th:hyp}} has to be compared to
\cite[Theorem 6]{barthe-cattiaux-roberto}. When $F=\log $,
\cite[Theorem 6]{barthe-cattiaux-roberto} asserts that
$\Vert P_{t} f \Vert _{\widetilde{q}(t)} \leq \Vert f \Vert _{2}$ with
$\widetilde{q}(t)=2e^{\rho t}$ which is off by a factor of $2$ in the exponential
(though capturing the exponential character of the $\mathbb{L}_{p}$ scale).

Furthermore, for $F(x):=\log (1+x)^{\beta }- \log (2)^{\beta }$,
$\beta \in (0,1)$,
\cite[Theorem 6 and Corollary 34]{barthe-cattiaux-roberto} does not give
an hypercontractivity property, but only hyper-boundedness (see section~\ref{sec:hb} below for more on hyper-boundedness). One of the main difference
comes from the fact that in \cite{barthe-cattiaux-roberto} the authors
deals with an explicit family of Young functions which imposes in some
situation stronger assumptions. This happens for the second assumption
of \cite[Theorem 6]{barthe-cattiaux-roberto} which reads in our setting
as
$\Phi _{t}(x)F(x^{2}) \leq \ell (t) \Phi _{t}( F(\Phi _{t}(x))) + m$. We
do not need such an assumption here.

To conclude with the comparison between the two theorems, we observe that
the first assumption of \cite[Theorem 6]{barthe-cattiaux-roberto} is implied
by $x \mapsto xF(x^{2})$ convex, see Remark~{\ref{rem:x2}} above and
\cite[Proposition 7]{barthe-cattiaux-roberto}.

Notice that, for the following smooth version of $|x|^{\alpha }$,
$\alpha \in (1,2)$,
\begin{equation*}
u_{\alpha }(x)=\left \{
\begin{array}{l@{\quad }l}
|x|^{\alpha }& \text{for } |x| >1
\\
\frac{\alpha (\alpha -2)}{8} x^{4} + \frac{\alpha (4-\alpha )}{4}x^{2}
+ (1-\frac{3}{4} \alpha + \frac{1}{8} \alpha ^{2}) & \text{for } |x| \leq 1
\end{array}
\right .
\end{equation*}
it has been proved in \cite[Proposition 33]{barthe-cattiaux-roberto} that
the probability measure
$d\mu _{\alpha }^{n}(x)=\prod _{i=1}^{n} Z_{\alpha }^{-1}e^{-u_{\alpha }(x_{i})}dx_{i}$
on $\mathbb{R}^{n}$ satisfies {\eqref{eq:FsobGross}} with
$F(x)=\log (1+x)^{\beta }- \log (2)^{\beta }$ with
$\beta =2(1-\frac{1}{\alpha })$, $\Phi _{0}(x)=x^{2}$ and some constant
$c=c(\alpha )>0$ (that does not depend on $n$). This in turn leads to an
hypercontractivity property for the standard Orlicz family built from
$F$, $\Phi _{0}$ and any $\lambda $ satisfying
$\lambda '(t) \leq 1/c$.
\end{remark}

\subsection{Perturbation of Orlicz families and hypercontractivity}
\label{sec3.4}

In this next section we show how to translate the hypercontractivity property
from one family of Young functions to another.

\begin{proposition}
Let $(\Psi _{t})_{t \geq 0}$ and $(\Phi _{t})_{t \geq 0}$ be two families
of Young functions and $(P_{t})_{t \geq 0}$ be a linear semi-group acting
on a set of functions $\mathcal{A}$ onto itself. Assume that for some
$t \geq 0$,%
\xmlpar $(i)$ any $f \in {\mathcal{A}}$ satisfies
$\Vert P_{t} f \Vert _{\Psi _{t}} \leq \Vert f \Vert _{\Psi _{0}}$,
\xmlpar $(ii)$ the function $\Psi _{t}^{-1} \circ \Phi _{t}$ is convex, satisfies
$\Psi _{t}^{-1} \circ \Phi _{t} \leq \Psi _{0}^{-1} \circ \Phi _{0}$ and
$\Psi _{t}^{-1} \circ \Phi _{t} (\mathcal{A}) \subset \mathcal{A}$,%
\xmlpar $(iii)$ for any function $f \in {\mathcal{A}}$, any convex function
$F$, $F(P_{t} f) \leq P_{t}(F(f))$.

Then, any $f \in {\mathcal{A}}$ satisfies
\begin{equation*}
\Vert P_{t} f \Vert _{\Phi _{t}} \leq \Vert f \Vert _{\Phi _{0}} .
\end{equation*}
\end{proposition}

\begin{proof}
By definition of the norm and Jensen's inequality given in $(iii)$, together
with $(ii)$, we have
\begin{align*}
1 & = \int \Phi _{t}\left (
\frac{P_{t} f}{\Vert P_{t} f \Vert _{\Phi _{t}}} \right )d\mu = \int
\Psi _{t} \circ \Psi _{t}^{-1} \circ \Phi _{t} \left ( P_{t}
\frac{f}{\Vert P_{t} f\Vert _{\Phi _{t}}} \right ) d\mu
\\
& \leq \int \Psi _{t} \left (P_{t} \Psi _{t}^{-1} \circ \Phi _{t}
\left ( \frac{f}{\Vert P_{t} f\Vert _{\Phi _{t}}} \right ) \right ) d
\mu .
\end{align*}
This implies by definition of the norm and the hypercontractivity for the
family $\Psi _{t}$ (given in $(i)$) that
\begin{equation*}
1 \leq \Vert P_{t} \Psi _{t}^{-1} \circ \Phi _{t} \Big (
\frac{f}{\Vert P_{t} f\Vert _{\Phi _{t}}} \Big ) \Vert _{\Psi _{t}}
\leq \Vert \Psi _{t}^{-1} \circ \Phi _{t} \Big (
\frac{f}{\Vert P_{t} f\Vert _{\Phi _{t}}} \Big ) \Vert _{\Psi _{0}} .
\end{equation*}
It follows that
$1 \leq \int \Psi _{0} \circ \Psi _{t}^{-1} \circ \Phi _{t} \Big (
\frac{f}{\Vert P_{t} f\Vert _{\Phi _{t}}} \Big ) d\mu $. Hence by point
$(ii)$,
$1 \leq \int \Phi _{0} \Big (
\frac{f}{\Vert P_{t} f\Vert _{\Phi _{t}}} \Big ) d\mu $. In turn,
$\Vert f \Vert _{\Phi _{0}} \geq \Vert P_{t} f \Vert _{\Phi _{t}}$. This
ends the proof.
\end{proof}

\begin{example}
Assume that $\Phi _{t} = \Psi _{t} \circ F$ for a fixed Young function
$F$. Then hypotheses $(ii)$ and $(iii)$ are automatically satisfied. For
instance, it is known that the linear semi-group with diffusion generator
$L=\Delta -\nabla U \nabla $ with $\mathrm{Hess}(U) \geq \rho >0$ satisfies
log-Sobolev inequality with constant $2/\rho $ and in turn is hypercontractive
in the sense that
\begin{equation*}
\Vert P_{t} f \Vert _{q(t)} \leq \Vert f \Vert _{2} , \qquad \qquad
\text{with } q(t):=1+e^{(4/\rho ) t}.
\end{equation*}
Now let $\Psi _{t}(x)=|x|^{q(t)}$. Hence, for any Young function $F$, the
previous proposition asserts that
\begin{equation*}
\Vert P_{t} f \Vert _{F^{q(t)}} \leq \Vert f \Vert _{F^{2}} .
\end{equation*}
Similarily from \cite{barthe-cattiaux-roberto} we learn that the semi-group
associated to $L= \Delta -\nabla U \nabla $ with $U(x)=|x|^{\alpha }$ (more
precisely a smoothed version of $|x|^{\alpha }$),
$1 \leq \alpha \leq 2$, is hypercontractive in the Orlicz' spaces family
$\mathbb{L}_{\Phi _{t}}$ with
$\Phi _{t} = x^{2} e^{ct \log (1+|x|^{2})^{2(1-\frac{1}{\alpha })}}$ for some
constant $c$. It follows that for any Young function $F$, the semi-group
$(P_{t})_{t \geq 0}$ is also hypercontractive in the Orlicz' spaces family
$\mathbb{L}_{\Psi _{t}}$ with
$\Psi _{t} = F(x)^{2} e^{ct \log (1+|F(x)|^{2})^{2(1-\frac{1}{\alpha })}}$.
\end{example}

\subsection{Hypercontractivity versus hyper-boundedness}
\label{sec:hb}

In this section, we deal with perturbation arguments that allows one to
get some hyper-boundedness property starting from hypercontractivity.

\begin{theorem}
Let $(\Phi _{t})_{t \geq 0}$ and $(\Psi _{t})_{t \geq 0}$ be two standard
Orlicz families built respectively from $F$ and $\widetilde{F}$,
$\Phi _{0}$ and $\lambda $, both satisfying the hypotheses of Definition~{\ref{def:sof}}.

Assume that for any $\varepsilon >0$ there exists
$D(\varepsilon )\geq 0$ such that all $x \geq 0$ satisfy
\begin{equation*}
\widetilde{F}(x) \leq \varepsilon F(x) + D(\varepsilon ) .
\end{equation*}
Suppose also that for any $f$ and any $t \geq 0$, it holds
\begin{equation*}
\Vert P_{t} f \Vert _{\Phi _{t}} \leq \Vert f \Vert _{\Phi _{0}} .
\end{equation*}

Then, for any $s_{2} \geq s_{1} \geq 0$, any $t \geq 0$, any
$\mathcal{C}^{1}$ increasing function
$q : \mathbb{R}^{+} \to \mathbb{R}^{+}$ with $q(0)=s_{1}$ and
$q(t)=s_{2}$, it holds
\begin{equation*}
\Vert P_{t} f \Vert _{\Psi _{s_{2}}} \leq \Vert f \Vert _{\Psi _{s_{1}}}
\exp \left \{  \int _{0}^{t} q'(u) \lambda '(q(u)) D \left (
\frac{\lambda '(0)}{q'(u)\lambda '(q(u))} \right ) du\right \} ,
\;\; \forall f \in \mathbb{L}_{\Psi _{s_{1}}} .
\end{equation*}
\end{theorem}

\begin{proof}
Our aim is to use the hypercontractivity property in Orlicz spaces
$\mathbb{L}_{\Phi _{t}}$ together with {\ref{th:hyp}} to get a functional
inequality involving the Young functions $\Phi _{t}$, and then use the
assumption on $F$ and $\widetilde{F}$ to get a similar inequality for
$\widetilde{F}$.

Fix $s_{2} \geq s_{1} \geq 0$, $t \geq 0$ and a $\mathcal{C}^{1}$ increasing
function $q : \mathbb{R}^{+} \to \mathbb{R}^{+}$ with $q(0)=s_{1}$ and
$q(t)=s_{2}$. Fix $f \in \mathbb{L}_{\Psi _{s_{1}}}$ and let
$N(u):=\Vert P_{u} f \Vert _{\Psi _{q(u)}}$ and $g:=P_{u} f/N(u)$. Applying
Lemma~{\ref{lem:technical}} to
$\widetilde{\Psi }(t,x):=\Phi _{q(t)}(x)$, and observing that
$\frac{\partial }{\partial t}\widetilde{\Psi }(t,x):=q'(t)\dot{\Psi }_{q(t)}(x)$,
we get
\begin{align*}
\frac{N'(u)}{N(u)} \int g \Psi _{q(u)}' \left ( g \right ) d\mu & = q'(u)
\int \dot{\Psi }_{q(u)} \left ( g \right ) d\mu - \int \Psi _{q(u)}''
\left ( g \right ) \left | \nabla g \right |^{2} d\mu .
\end{align*}

Since $\Psi _{q(u)}$ is a nice Young function,
$x \Psi _{q(u)}'(x) \geq \Psi _{q(u)}(x)$ for any $x \geq 0$, any
$u$. Therefore $\int g \Psi _{q(u)}' \left ( g \right ) d\mu \geq 1$ and
in turn, when $N'(u) \geq 0$, we have
\begin{equation*}
\frac{N'(u)}{N(u)} \leq q'(u) \int \dot{\Psi }_{q(u)} \left ( g\right )
d\mu - \int \Psi _{q(u)}'' \left ( g \right ) \left | \nabla g\right |^{2}
d\mu .
\end{equation*}
Now, thanks to {\ref{th:hyp}}, the hypercontractivity assumption
guarantees that
\begin{equation*}
\Vert f \Vert _{\Phi _{0}}^{2} \int \Phi _{0} \left (
\frac{f}{\Vert f \Vert _{\Phi _{0}}} \right ) F \left ( \Phi _{0}
\left ( \frac{f}{\Vert f \Vert _{\Phi _{0}}} \right ) \right ) d \mu
\leq \frac{1}{\lambda '(0)} \int \Phi _{0}'' \left (
\frac{f}{\Vert f \Vert _{\Phi _{0}}} \right ) |\nabla f|^{2} d\mu .
\end{equation*}
By our assumption, item $(iii)$ of Lemma~{\ref{lem:sof}} (recall that
$\Psi _{0}=\Phi _{0}$), we have
\begin{equation*}
\dot{\Psi }_{0} = \lambda '(0) \Phi _{0} \widetilde{F}(\Phi _{0})
\leq \lambda '(0) \varepsilon \Phi _{0} F(\Phi _{0}) + \lambda '(0)D(
\varepsilon )\Phi _{0} .
\end{equation*}
Therefore, for any $f$ with $\|f\|_{\Phi _{0}}=1$,
\begin{align*}
\int \dot{\Psi }_{0} \left ( f \right ) d\mu & \leq \lambda '(0)
\varepsilon \int \Phi _{0} \left ( f \right ) F \left ( \Phi _{0}
\left ( f \right ) \right ) d\mu + \lambda '(0) D(\varepsilon )
\\
& \leq \varepsilon \int \Phi _{0}'' \left ( f \right ) \left |
\nabla f \right |^{2} d\mu + \lambda '(0) D(\varepsilon ) \\
& =
\varepsilon \int \Psi _{0}'' \left ( f \right ) \left | \nabla f
\right |^{2} d\mu + \lambda '(0) D(\varepsilon ) .
\end{align*}
Recall the isometry
\begin{equation*}
\begin{array}{r@{\quad }c@{\quad }l}
I_{s,t} : \mathbb{L}_{\Psi _{t}} & \rightarrow & \mathbb{L}_{\Psi _{s}}
\\
f & \mapsto & \Vert f \Vert _{\Psi _{t}} \Psi _{s}^{-1} \circ \Psi _{t}
\left ( \frac{f}{\Vert f \Vert _{\Psi _{t}}} \right )
\end{array}
\end{equation*}
from Proposition~{\ref{prop:move}} that we may use with $s=0$ and
$t=q(u)$. The previous inequality applied to $I_{0,q(u)}(f)$ ensures that
for any $f$ with
$\| f \|_{\Psi _{q(u)}}=\|I_{0,q(u)}(f)\|_{\Psi _{0}}=\|I_{0,q(u)}(f)
\|_{\Phi _{0}}=1$, it holds
\begin{equation*}
\int \! \dot{\Psi }_{0} \circ \Psi _{0}^{-1} \circ \Psi _{q(u)}(f) d\mu
\leq 
\varepsilon \int \! \frac{\Psi _{0}''}{{\Psi _{0}'}^{2}} \circ
\Psi _{0}^{-1} \circ \Psi _{q(u)} (f) \left | \nabla f \right |^{2} 
{\Psi'}^2_{q(u)}(f) d\mu + \lambda '(0) D(\varepsilon ) .
\end{equation*}
It follows from items $(iii)$ and $(iv)$ of Lemma~{\ref{lem:sof}} that
\begin{equation*}
\frac{\lambda '(0)}{\lambda '(q(u))} \int \dot{\Psi }_{q(u)} (f) d\mu
\leq \varepsilon \int \Psi _{q(u)}'' (f) \left | \nabla f \right |^{2}
d\mu + \lambda '(0) D(\varepsilon ) .
\end{equation*}
This leads to
\begin{align*}
\frac{N'(u)}{N(u)} 
& \leq 
\frac{q'(u)\lambda '(q(u))}{\lambda '(0)}
\left ( \varepsilon \int \Psi _{q(u)}'' (f) \left | \nabla f \right |^{2}
d\mu + \lambda '(0) D(\varepsilon ) \right ) \\
& \quad - \int \Psi _{q(u)}'' (f)
\left | \nabla f \right |^{2} d\mu
\end{align*}
for any $u$ such that $N'(u) \geq 0$. The latter being valid
for any $\varepsilon >0$, choose $\varepsilon $ such that
$q'(u)\lambda '(q(u)) \varepsilon = \lambda '(0)$. Therefore, provided
that $N'(u)\geq 0$ it holds
\begin{equation*}
\frac{N'(u)}{N(u)} \leq q'(u) \lambda '(q(u)) D \left (
\frac{\lambda '(0)}{q'(u)\lambda '(q(u))} \right ) .
\end{equation*}
This bound trivially holds when $N'(u) <0$. Hence
\begin{align*}
\log \Vert P_{t} f \Vert _{\Psi _{s_{2}}} - \log \Vert P_{0} f \Vert _{
\Psi _{s_{1}}} & = \int _{0}^{t} \frac{d}{du} \log \Vert P_{u} f
\Vert _{\Psi _{q(u)}} du = \int _{0}^{t} \frac{N'(u)}{N(u)} du
\\
& \leq \int _{0}^{t} q'(u) \lambda '(q(u)) D \left (
\frac{\lambda '(0)}{q'(u)\lambda '(q(u))} \right ) du .
\end{align*}
The result follows.
\end{proof}

As an example of application, consider, for $\beta \in (0,1]$,
$F_{\beta }(x)=(\log (1+x))^{\beta }- (\log 2)^{\beta }$. It is not difficult
to check that for any $\varepsilon >0$ and any $\beta ' < \beta $, it holds
\begin{equation*}
F_{\beta '}(x) \leq \varepsilon F_{\beta }(x) + D(\varepsilon )
\end{equation*}
with
\begin{equation*}
D(\varepsilon ):= -(\log 2)^{\beta '} + \varepsilon (\log 2)^{\beta }+
\left ( \frac{\beta '}{\beta } \right )^{\frac{\beta '}{\beta - \beta '}} \frac{\beta - \beta '}{\beta } \left (
\frac{1}{\varepsilon } \right )^{\frac{\beta '}{\beta - \beta '}} .
\end{equation*}
Now let
\begin{equation*}
u_{\alpha }(x)=\left \{
\begin{array}{l@{\quad }l}
|x|^{\alpha }& \text{for } |x| >1
\\
\frac{\alpha (\alpha -2)}{8} x^{4} + \frac{\alpha (4-\alpha )}{4}x^{2}
+ (1-\frac{3}{4} \alpha + \frac{1}{8} \alpha ^{2}) & \text{for } |x| \leq 1
\end{array}
\right .
\end{equation*}
be a smooth version of $|x|^{\alpha }$, $\alpha \in (1,2)$. Define the probability
measure
$d\mu _{\alpha }^{n}(x)=\prod _{i=1}^{n} Z_{\alpha }^{-1}e^{-u_{\alpha }(x_{i})}dx_{i}$
on $\mathbb{R}^{n}$. As already mentioned in Remark~{\ref{rem:go}}, it follows
from \cite[Proposition 33]{barthe-cattiaux-roberto} that Inequality {\eqref{eq:FsobGross}} holds for $F_{\beta }$, $\Phi _{0}(x)=x^{2}$ and some
$c=c(\alpha )>0$ and therefore that the semi-group
$(P_{t})_{t \geq 0}$ associated to $\mu _{\alpha }^{n}$ is hypercontractive
along the standard Orlicz family $(\Phi _{t})_{t \geq 0}$ built from
$F_{\beta }$, $\Phi _{0}$ and any $\lambda $ satisfying
$\lambda '(t) \leq 1/c$. Fix for simplicity $\lambda (t)=t/c$.

Consider the standard Orlicz families $(\Psi _{t})_{t \geq 0}$ built from
$F_{\beta '}$, $\Phi _{0}(x)=x^{2}$ and $\lambda $.

The previous theorem shows that (for $s_{1}=0$ and $s_{2}=s$)
\begin{equation*}
\Vert P_{t} f \Vert _{\Psi _{s}} \leq m(t) \Vert f \Vert _{2}
\end{equation*}
where
\begin{equation*}
m(t) = \inf _{q} \exp \left \{  \frac{1}{c}\int _{0}^{t} q'(u) D
\left ( \frac{1}{q'(u)} \right ) du\right \}
\end{equation*}
where the infimum is running over all increasing
$q \colon [0,t] \to \mathbb{R}_{+}$ with $q(0)=0$ and $q(t)=s$. We stress
that the Luxembourg norm is computed here with reference measure
$\mu _{\alpha }^{n}$.

One has
\begin{equation*}
q'(u) D \left ( \frac{1}{q'(u)} \right ) =-q'(u)(\log 2)^{\beta '} + (
\log 2)^{\beta }+ C_{\beta ,\beta '} (q'(u))^{\frac{\beta }{\beta -\beta '}}
\end{equation*}
where we set, for simplicity,
$C_{\beta ,\beta '}:= \left ( \frac{\beta '}{\beta } \right )^{\frac{\beta '}{\beta - \beta '}} \frac{\beta - \beta '}{\beta }$. Therefore,
\begin{align*}
\inf _{q} \left \{  \int _{0}^{t} q'(u) D \left ( \frac{1}{q'(u)}
\right ) du \right \}  
& = 
-s (\log 2)^{\beta '} 
+ t (\log 2)^{\beta } \\
&\quad 
+ C_{\beta ,\beta '} \inf _{q} \left \{  \int _{0}^{t} (q'(u))^{\frac{\beta }{\beta -\beta '}} du \right \}  .
\end{align*}
Since $\beta /(\beta -\beta ') \geq 1$, by Holder's inequality (and equality
cases in Holder's inequality), it is easy to see that
\begin{equation*}
\inf _{q} \left \{  \int _{0}^{t} (q'(u))^{\frac{\beta }{\beta -\beta '}} du
\right \}  = s^{\frac{\beta }{\beta -\beta '}}t^{-
\frac{\beta '}{\beta -\beta '}}.
\end{equation*}
As a conclusion
\begin{equation*}
m(t)=\exp \left \{  \frac{1}{c} \left ( -s (\log 2)^{\beta '} + t (
\log 2)^{\beta }+ C_{\beta ,\beta '} s^{\frac{\beta }{\beta -\beta '}}t^{-
\frac{\beta '}{\beta -\beta '}} \right ) \right \}  .
\end{equation*}
Note in particular that the factor in the exponential explodes for a fixed
$t$, when $s$ goes to infinity. This must be since the semi-group associated
to $\mu _{\alpha }^{n}$ can not be ultracontractive.

\section{Contraction property for inhomogeneous Markov semi-groups}
\label{sec:inhomogeneous}

In this section we deal with the time-dependent diffusion operators
$L_{t}:=\Delta - \nabla V_{t} \cdot \nabla $, $t \geq 0$, on
$\mathbb{R}^{n}$, with $V_{t}$ smooth enough and such that
$\int e^{-V_{t}}=1$. Recall that the associated semi-group
$(P^{(t)}_{s})_{s \geq 0}$ is reversible with respect to the probability
measure $\mu _{t}(dx):=e^{-V_{t}(x)}dx$. All along this section Luxembourg
norms are understood with respect to $\mu _{t}$. We may omit
such a dependence when not needed and write otherwise
$\| \cdot \|_{\Phi _{t},\mu _{t}}$.

In order to obtain contraction bounds for $P^{(t)}_{t}f$, one could try
to use the following natural simple strategy. For $t$ fixed, one may assume
that $\mathrm{Hess}(V_{t}) \geq \rho _{t}$ for some $\rho _{t}>0$ so that
Gross' theorem applies and leads to
\begin{equation*}
\| P_{s}^{(t)} f \|_{q_{t}(s),\mu _{t}} \leq \| f \|_{2,\mu _{t}}
\qquad s,t \geq 0,
\end{equation*}
with, say, $q_{t}(s) \leq 1 + e^{(4/\rho _{t})s}$ where we set
$\| g \|_{q,\mu _{t}} := \left ( \int |g|^{q} d\mu _{t} \right )^{\frac{1}{q}}$ for the $\mathbb{L}_{q}$ norm of $g$ with respect to
$\mu _{t}$ (we choose $q_{t}(0)=2$ for simplicity). Applying the latter
at time $s=t$ leads to
\begin{equation*}
\| P_{t}^{(t)} f \|_{q_{t}(t),\mu _{t}} \leq \| f \|_{2,\mu _{t}} .
\end{equation*}
Now observe that the latter might be very weak if $\rho _{t}$ is small.
Moreover and more essentially one would like to deal with a norm on the
right hand side independent of $t$ (say related to $\mu _{0}$). Before
achieving this program, by means of Lemma~{\ref{lem:technical}}, let us end
this introduction with an easy (and very specific) example of inhomogeneous
Markov semi-group whose hypercontractivity property can be derived from
the results of the previous section.

Consider for instance $V_{t} = U_{1}$ for $t \in [0,T]$ and
$V_{t}=U_{2}$ for $t >T$ where $U_{1}$ and $U_{2}$ are associated to hypercontractivity
properties in Orlicz spaces $\mathbb{L}_{\Phi _{t}^{(1)}}$ and
$\mathbb{L}_{\Phi _{t}^{(2)}}$ respectively. Then, we can argue that
$\| P_{t}^{(t)}f \|_{\Phi _{t}^{(1)}} \leq \| P_{s}^{(s)}f \|_{\Phi _{s}^{(1)}}$,
for any $s \leq t \leq T$, and then
$\| P_{t}^{(t)}f \|_{\Phi _{t}^{(2)}} \leq \| P_{s}^{(s)}f \|_{\Phi _{s}^{(2)}}$
for any $T < s \leq t$. Therefore, if the two families of Orlicz spaces
coincide at time T, \textit{i.e.} $\Phi _{T}^{(1)}=\Phi _{T}^{(2)}$, and
potentially modulo some extra assumptions on the Young functions
$\Phi _{t}^{(1)}$, $\Phi _{t}^{(2)}$, if we set
$\Phi _{t}=\Phi _{t}^{(1)}$ for $t \in [0,T]$ and
$\Phi _{t}:=\Phi _{t}^{(2)}$ for $t \geq T$, one has
$\| P_{t}^{(t)}f \|_{\Phi _{t}} \leq \| P_{s}^{(s)}f \|_{\Phi _{s}}$ for
all $s \leq t$. As already mentioned this is however very specific and
somehow artificially inhomogeneous. We would like to deal with examples
of potentials $V_{t}$ that evolve all along the time $t$.

In the next section we will deal under the restricted hypothesis of the
log-Sobolev inequality {\eqref{eq:Fsob3}} related to the Orlicz family
$\Phi _{t}(x)=|x|^{q(t)}$ ($\mathbb{L}_{p}$ scale). This makes the presentation
more precise and easier by reducing some technicalities. However, it already
encompasses many of the difficulties. The last section (that comes after)
will finally deal with a more general setting.

We stress that the results below are a first step in the understanding
of contraction properties for inhomogeneous Markov semi-groups. Many remain
to be discovered and we believe that our investigations open new lines
of research with possible application, as mentioned in the introduction,
to non-linear parabolic time dependent problems (in infinite dimension).

\subsection{$\mathbb{L}_{p}$-scales}
\label{sec4.1}

In order to give the flavor of what is happening in the inhomogeneous setting
(and avoid some technicalities), in this section we may only deal with
contractivity properties in $\mathbb{L}_{p}$-scales. Recall that
$X_{-}=max(-X,0)$ stands for the negative part.

\begin{theorem}
\label{th:main1}
Consider the inhomogeneous diffusion operator $L_{t}$ as above. Set
$a_{t}:=\| (\dot{V}_{t})_{-} \|_{\infty }$,
$b_{t}:=\| |\nabla \dot{V}_{t} |\|_{\infty }$,
$c_{t}:=\| (\nabla V_{t} \cdot \nabla \dot{V}_{t} - \Delta \dot{V}_{t})_{-}
\|_{\infty }$ and assume that $a_{t},b_{t},c_{t} < \infty $ for all
$t \geq 0$. Assume also that, for all $t \geq 0$ there exists
$\rho _{t} \in \mathbb{R}$ such that
$\mathrm{Hess}(V_{t}) \geq \rho _{t}$. Finally, assume that there exists
$\bar{\rho }_{t}>0$ such that the following log-Sobolev inequality holds
%
\begin{equation}
\label{eq:Fsob3}
\int f^{2} \log (f^{2}) d\mu _{t} \leq \bar{\rho }_{t} \int |\nabla f|^{2}
d\mu _{t} ,
\end{equation}
for all $f$ with $\int f^{2} d\mu _{t} =1$ for which the right hand side
is well defined. Then, for any
$f \colon \mathbb{R}^{n} \to \mathbb{R}_{+}$ smooth enough, any
$p >1$ and any $s \leq t$, it holds
\begin{equation*}
\| P_{t}^{(t)}f\|_{\Phi _{t},{\mu _{t}}} \leq m(s,t) \| P_{s}^{(s)}f
\|_{\Phi _{s},{\mu _{s}}}
\end{equation*}
where $\Phi _{t}(x)=|x|^{q(t)}$,
$q(t)=1+(p-1)\exp \{\int _{0}^{t} (2/\bar{\rho }_{s}) ds \}$,
$t \geq 0$ and
\begin{equation*}
m(s,t):=\exp \left \{  \int _{s}^{t} \frac{a_{u}}{q(u)} + uc_{u} + b_{u}^{2}
\frac{1-e^{-\rho _{u} u}}{2\rho _{u}} (q(u)-1) du \right \}  .
\end{equation*}
\end{theorem}

\begin{remark}
As already mentioned, one possible criterion for the log-Sobolev inequality {\eqref{eq:Fsob3}} to hold, is $\mathrm{Hess}(V_{t}) \geq \rho _{t}>0$ (as
a matrix), which implies $\bar{\rho }_{t} \leq \frac{2}{\rho _{t}}$. Alternatively,
as will be used below, one can apply a perturbation argument
\emph{\`{a} la} Holley \& Stroock \cite{holleystroock87}.

If $V_{t}$ does not depend on $t$ then $m(s,t)=1$ and
$q(t)=1+(p-1)e^{(2/\bar{\rho }) t}$, which is Gross' theorem off by a factor
of $2$ in the exponential (see the introduction). This is coming from a
technical computation that uses Cauchy-Schwarz' inequality. One can actually
improve this and get
$q(t)=1+(p-1)\exp \{(2-\varepsilon )\int _{0}^{t} (2/ \bar{\rho }_{s}) ds
\}$, for any $\varepsilon >0$, but at the price of a factor $m(s,t)$ that
depends on $\varepsilon $, and that increases when $\varepsilon $ decreases.

Modulo such a factor 2, the above theorem can therefore be seen as an inhomogeneous
counterpart of Gross' theorem.
\end{remark}

\begin{example}
The above theorem contains some non trivial examples. For instance one
can consider potentials of the form
$V_{t}(x)=U(x) + \alpha (t) V(x) + {\gamma }(t)$ with $V$ unbounded
and ${\gamma }(t):=\log \int e^{-U-\alpha V} dx$ so that
$\mu _{t}$ indeed defines a probability measure.

Take for instance $U(x)=\frac{|x|^{2}}{2}$,
$\alpha \colon [0,\infty ) \to [0,\infty )$ non-decreasing and
$V(x)=(1+|x|^{2})^{\frac{\beta }{2}}$, with $\beta \in (0,1]$ (this is an
unbounded (time-dependent) perturbation of the standard Gaussian potential
$U$).

Then, $\dot{V}_{t}=\alpha '(t)(1+|x|^{2})^{\frac{\beta }{2}}$ so that
$a_{t}=0$;
$\nabla \dot{V}_{t} =\alpha '(t) \beta (1+|x|^{2})^{\frac{\beta }{2}-1}x$,
whence
$b_{t} = \alpha '(t) \beta \sqrt{
\frac{(1-\beta )^{1-\beta }}{(2-\beta )^{2-\beta }}}$ (which is understood
as its limit when $\beta =1$, namely $b_{t}=\alpha '(t)$ if
$\beta =1$). Using crude estimates, it is not difficult to prove that
$c_{t} \leq n^{2} \alpha '(t)(\alpha (t)+2)$. On the other hand (again
we omit details) $\mathrm{Hess}(V_{t}) \geq 1$ so that $\rho _{t}=1$ and {\eqref{eq:Fsob3}} holds with $\bar{\rho }_{t}=2$. {\ref{th:main1}} then implies that the corresponding  
inhomogeneous
semi-group
$(P_{t}^{(t)})_{t \geq 0}$ is hyper-bounded in the
$\mathbb{L}_{q(t)}$ scale, with $q(t)=1+(p-1)e^{t}$.

For $V(x)=\log (1+|x|^{2})$ one can easily see that $a_{t}=0$,
$b_{t}=\alpha '(t)$ and $c_{t}<\infty $. The issue is coming from estimating
$\bar{\rho }_{t}$. In fact $\mathrm{Hess}(V_{t})(x)$ is bounded below by
a positive matrix only outside a ball (of radius proportional to
$\alpha $). Therefore, one can write $V_{t}=H+R$, with $H$ strictly convex,
in the sense that $\mathrm{Hess}(H) \geq 1/2$, say, and $R$ is bounded.
Then Bakry-\'{E}mery criterion applies to the measure with density proportional
to $e^{-H}$, leading to a log-Sobolev constant at most $4$, and then we
use Holley-Stroock perturbation Lemma, see \textit{e.g.} \cite[Theorem 3.4.3]{ane}
to get Inequality {\eqref{eq:Fsob3}} with constant $\bar{\rho }_{t}$ at most
$4e^{\mathrm{Osc}(R)}$ (therefore potentially exponentially big in
$\alpha $) where $\mathrm{Osc}(R)=\sup R - \inf R$ is the oscillation of
$R$. {\ref{th:main1}} applies and finally leads to some contraction
property with $q(t) \to \infty $ for $\alpha $ bounded or slowly growing
to infinity (for instance $\alpha (t)=\log \log t$ would do).
\end{example}

\begin{proof}[Proof of {\ref{th:main1}}]
Let $\Phi _{t}(x):=|x|^{q(t)}$ with
$q \colon [0,\infty ) \to [0,\infty )$ increasing and satisfying
$q(0)>1$. Set
$N(t):= \| P_{t}^{(t)}f \|_{\Phi _{t} {, \mu _{t}}}$ for some
non-negative smooth $f$, and $g:=\frac{P_{t}^{(t)}f}{N}$ so that
$\int \Phi _{t}(g)d\mu _{t} =1$. From Lemma~{\ref{lem:technical}}, we have
\begin{align*}
N'&(t) \int g \Phi _{t}'(g) d\mu _{t} \leq N(t) \left (\int
\dot{\Phi }_{t}(g) d\mu _{t} - \int {\Phi _{t}''}(g) |
\nabla g|^{2} d\mu _{t} + a_{t} \right )
\\
& + b_{t} \int \int _{0}^{t} P_{t-s}^{(t)} f |\nabla P_{s}^{(t)}(
\Phi _{t}'(g))| dsd\mu _{t} + c_{t} \int \int _{0}^{t} P_{t-s}^{(t)} f
P_{s}^{(t)} (\Phi _{t}'(g)) ds d\mu _{t} .
\end{align*}
We observe that $x\Phi _{t}'(x)=q\Phi _{t}(x)$ so that
$\int g \Phi _{t}'(g) d\mu _{t} =q$. Also, by reversibility, the last term
of the latter satisfies
\begin{align*}
\int \int _{0}^{t} P_{t-s}^{(t)} f P_{s}^{(t)} (\Phi _{t}'(g)) ds d
\mu _{t} & = \int _{0}^{t} \int P_{s}^{(t)}(P_{t-s}^{(t)} f) \Phi _{t}'(g)
d\mu _{t} ds
\\
& = t N(t) \int g \Phi _{t}'(g) d\mu _{t} = tN(t)q(t).
\end{align*}
Since $\dot{\Phi }_{t}(x)=q'(t) |x|^{q} \log (|x|)$ and
$\Phi _{t}''(x) =q(q-1)|x|^{q-2}$, we get
\begin{align*}
q(t)\frac{N'(t)}{N(t)} & \leq \frac{q'(t)}{q(t)} \mathrm{Ent}_{\mu _{t}}(g^{q})
- q(q-1) \int g^{q-2} |\nabla g|^{2} d\mu _{t} +a_{t} + tc_{t} q(t)
\\
& \quad + \frac{b_{t}}{N(t)} \int \int _{0}^{t} P_{t-s}^{(t)} f |
\nabla P_{s}^{(t)}(\Phi _{t}'(g))| dsd\mu _{t} .
\end{align*}
The condition $\mathrm{Hess}(V_{t}) \geq \rho _{t}$ ensures that
$|\nabla P_{s}^{(t)} h| \leq e^{-\rho _{t} s} P_{s}^{(t)}(|\nabla h|)$
for all $s \geq 0$ and all $h$ (see \textit{e.g.} \cite{ane}[Proposition
5.4.5]). Hence, by reversibility
\begin{align*}
\int \int _{0}^{t} P_{t-s}^{(t)} f |\nabla P_{s}^{(t)}(\Phi _{t}'(g))|
dsd\mu _{t} & \leq \int \int _{0}^{t} e^{-\rho _{t} s} P_{t-s}^{(t)} f
P_{s}^{(t)}(|\nabla \Phi _{t}'(g)|) dsd\mu _{t}
\\
& = \int \int _{0}^{t} e^{-\rho _{t} s} P_{t}^{(t)} f \Phi _{t}''(g)|
\nabla g| dsd\mu _{t}
\\
& = N(t) \frac{1-e^{-\rho _{t} t}}{\rho _{t}} q(q-1) \int g^{q-1} |
\nabla g| d\mu _{t} .
\end{align*}
Using the inequality
$uv \leq \frac{1}{2 \varepsilon }u^{2} + \frac{\varepsilon }{2}v^{2}$ with
$\varepsilon =b_{t}\frac{1-e^{-\rho _{t} t}}{\rho _{t}}$,
$u=g^{\frac{q}{2}-1} |\nabla g|$ and $v= g^{\frac{q}{2}}$ we get
\begin{equation*}
\frac{1-e^{-\rho _{t} t}}{\rho _{t}} \int g^{q-1} |\nabla g| d\mu _{t}
\leq \frac{1}{2b_{t}}\int g^{q-2} |\nabla g|^{2} d\mu _{t} +
\frac{1}{2}b_{t}{\left (
\frac{1-e^{-\rho _{t} t}}{\rho _{t}} \right )^{2} } \int g^{q} d\mu _{t}
\end{equation*}
so that, since $\int g^{q} d\mu _{t} = \int \Phi _{t}(g)d\mu _{t}=1$,
\begin{align*}
q(t)\frac{N'(t)}{N(t)} & \leq \frac{q'(t)}{q(t)} \mathrm{Ent}_{\mu _{t}}(g^{q})
- \frac{q(q-1)}{2} \int g^{q-2} |\nabla g|^{2} d\mu _{t} +a_{t} + tc_{t}
q(t)
\\
& \quad + b_{t}^{2}\frac{1-e^{-\rho _{t} t}}{2\rho _{t}} q(t)(q(t)-1).
\end{align*}
Next, we observe that
$\int g^{q-2} |\nabla g|^{2} d\mu _{t} = \frac{4}{q^{2}} \int |
\nabla g^{\frac{q}{2}}|^{2} d\mu _{t}$. Hence, for
$q(t):=1+(p-1)\exp \{ \int _{0}^{t} (2/\bar{\rho }_{s}) ds \}$ which satisfies
$\frac{2(q-1)}{q'}=\bar{\rho }_{t}$, we are guaranteed by {\eqref{eq:Fsob3}} that
\begin{align*}
\frac{q'(t)}{q(t)} \mathrm{Ent}_{\mu _{t}}(g^{q}) & -
\frac{q(q-1)}{2} \int g^{q-2} |\nabla g|^{2} d\mu _{t}
\\
& = \frac{q'(t)}{q(t)} \left ( \mathrm{Ent}_{\mu _{t}}(g^{q}) -
\frac{2(q-1)}{q'(t)} \int |\nabla g^{\frac{q}{2}}|^{2} d\mu _{t}
\right )
\\
& = \frac{q'(t)}{q(t)} \left ( \mathrm{Ent}_{\mu _{t}}(g^{q}) -
\bar{\rho }_{t} \int |\nabla g^{\frac{q}{2}}|^{2} d\mu _{t} \right )
\leq 0 .
\end{align*}
It follows that
\begin{align*}
\frac{N'(t)}{N(t)} & \leq \frac{a_{t}}{q(t)} + tc_{t} + b_{t}^{2}
\frac{1-e^{-\rho _{t} t}}{2\rho _{t}} (q(t)-1)
\end{align*}
which leads to the desired conclusion.
\end{proof}

\subsection{General result}
\label{sec4.2}

In this section we establish a more general result than {\ref{th:main1}} that allows one to deal with more general Orlicz
families, and not only the $\mathbb{L}_{p}$-scales. As a motivation, one
can consider for instance as above
$V_{t}(x)=U(x) + \alpha (t) V(x) + \gamma (t)$ with
$U(x) \simeq \frac{|x|^{\alpha }}{\alpha }$ (for large $|x|$),
$\gamma (t):=\log \int e^{-U-\alpha V} dx$ and
$\Phi _{t}(x)=x^{2}e^{ct F(x)}$, with $F(x) \simeq \log (x)^{\beta }$ (for
large $x$). This corresponds, with a proper choice of $V$, to a generalization
of the hypercontractivity property proved in \cite{BCR07} in the homogeneous
setting (recall the introduction, see also Remark~{\ref{rem:go}}.

\begin{theorem}
\label{th:main2}
Consider the inhomogeneous diffusion operator $L_{t}$ as above. Assume
that for all $t \geq 0$,
$b_{t}:=\||\nabla \dot{V}_{t}| \|_{\infty }< \infty $ and that there exists
$\rho _{t} \in \mathbb{R}$ such that
$\mathrm{Hess}(V_{t}) \geq \rho _{t}$ (as a matrix).

Let $(\Phi _{t})_{t \geq 0}$ be a family of Young functions satisfying
$\Phi _{t}(x) \leq x \Phi _{t}'(x) \leq B_{t} \Phi _{t}(x)$,
${\Phi _{t}'}^{2} \leq C_{t} \Phi _{t} \Phi _{t}''$ and
$x^{2} \Phi _{t}''(x) \leq D_{t} \Phi _{t}(x) + E_{t}$ for all
$x \geq 0$ and some constants $B_{t}, C_{t}, D_{t}, E_{t}$.

Assume that for all $t \geq 0$ there exist $\delta _{t} \in [0,1)$ and
$F_{t} \in \mathbb{R}$ such that
$(\dot{V}_{t})_{-} \leq \frac{\delta _{t}}{4C_{t}} \left (|\nabla V_{t}|^{2}
- 2 \Delta V_{t}\right ) + F_{t}$.

Set
$W_{t}:=(\nabla V_{t} \cdot \nabla \dot{V}_{t} - \Delta \dot{V}_{t})_{-}$
and denote by $\bar{\rho }_{t} \in (0,\infty ]$ the best constant such that
for all $f$ with $\| f \|_{\Phi _{0}} =1$ it holds
%
\begin{equation}
\label{eq:Fsob2}
\int \dot{\Phi }_{t}(f) d\mu _{t} \leq \bar{\rho }_{t} \int \Phi _{t}''(f)
|\nabla f|^{2} d\mu _{t} .
\end{equation}

Finally, assume either that
\xmlpar $(i)$ $c_{t}:=\|W_{t}\|_{\infty }< \infty $ and
$\bar{\rho }_{t} < 1-\delta _{t}$;%
\xmlpar or
\xmlpar $(ii)$
$c'_{t}:= \max \left ( 2\||\nabla W_{t}|\|_{\infty }/b_{t} ,\sup _{x : W_{t}(x)
\neq 0}\left (\frac{L_{t}W_{t}}{W_{t}}-\rho _{t}\right )_{-} \right ) <
\infty $ and that for all $t \geq 0$ there exists
$\delta '_{t} \in [0,1)$ and $F'_{t} \in [0,\infty )$ such that
$\delta '_{t} \int _{0}^{t} e^{(c' _{s}-\rho _{s})s} ds <1$ and
$W_{t} \leq \frac{\delta '_{t}}{4B_{t}C_{t}} \left (|\nabla V_{t}|^{2}
- 2 \Delta V_{t}\right ) + F'_{t}$.

Then, for any $f \colon \mathbb{R}^{n} \to \mathbb{R}_{+}$ smooth enough,
it holds
\begin{equation*}
\| P_{t}^{(t)}f\|_{\Phi _{t} {, \mu _{t}}} \leq m(s,t) \| P_{s}^{(s)}f
\|_{\Phi _{s}{, \mu _{s}}}
\end{equation*}
where under assumption $(i)$,
\begin{equation*}
m(s,t)= \exp \left \{  \int _{s}^{t} F_{u} + \left ( b_{u}
\frac{1-e^{-\rho _{u} u}}{\rho _{u}} \right )^{2}
\frac{D_{u}+E_{u}}{2(1-\delta _{u} - \bar{\rho }_{u})} + c_{u}B_{u} u du
\right \}  ,
\end{equation*}
and under assumption $(ii)$,
\begin{align*}
m(s,t) = \exp & \left\{  \int _{s}^{t} \left ( \int _{0}^{u} e^{(c'_{v}-
\rho _{v})v}dv \right )^{2}
\frac{b_{u}(D_{u}+E_{u})}{2(1-\delta _{u} - \bar{\rho }_{u} - \delta '_{u} \int _{0}^{u} e^{(c'_{v}-\rho _{v})v}dv)} \right.
\\
& \qquad \qquad  + \left. \int _{0}^{u} e^{(c'_{v}-\rho _{v})v}dv
\left ( B_{u}F'_{u} + F_{u} \right ) du \right\}  .
\end{align*}
\end{theorem}

\begin{remark}
Observe that, when $\delta _{t}=0$, the assumption
$$
(\dot{V}_{t})_{-} \leq \frac{\delta _{t}}{C_{t}} \left (|\nabla V_{t}|^{2}
- 2 \Delta V_{t}\right ) + D_{t}
$$ 
amounts to
$a_{t} := \|(\dot{V}_{t})_{-}\|_{\infty }< \infty $ which is the assumption
that we used in {\ref{th:main1}}. Also, it might be that
$\dot{V}_{t} \geq 0$ so that, in that case, one chooses
$\delta _{t}=D_{t}=0$.

Observe also that the first inequality in the assumption
$\Phi _{t}(x) \leq x \Phi _{t}'(x) \leq B_{t} \Phi _{t}(x)$ is satisfied
by all Young functions, while the second inequality is a consequence of
the $\Delta _{2}$-condition.

Finally we observe that, although we weakened most of the hypotheses of
{\ref{th:main1}}, one key assumption one would like to remove/reduce
is $b_{t}=\||\nabla \dot{V}_{t}| \|_{\infty }< \infty $. Indeed one interesting
example one would like to deal with is for instance
$V_{t}(x)=(1-t)_{+}^{2}|x|^{2} + |x|^{\alpha }$, with
$\alpha \in (1,2)$, where we have a critical point
$t=1$ in which hypercontractivity property in $\mathbb{L}_{p}$ spaces is
replaced by a weaker property. Such an example is not covered by {\ref{th:main2}} since $b_{t}=\infty $.
\end{remark}

\begin{proof}[Proof of {\ref{th:main2}}]
We start as in the proof of {\ref{th:main1}}. Set
$N(t):= \| P_{t}^{(t)}f \|_{\Phi _{t}{,\mu _{t}}}$,
$g:=\frac{P_{t}^{(t)}f}{N}$ so that from Lemma~{\ref{lem:technical}} for
some non-negative smooth function $f$, it holds
%
\begin{align}
\label{eq:last2}
N'(t) \int g \Phi _{t}'(g) d\mu _{t} & \leq N(t) \left (\int
\dot{\Phi }_{t}(g) d\mu _{t} - \int \phi _{t}''(g) |\nabla g|^{2} d
\mu _{t} - \int \Phi _{t}(g) \dot{V}_{t} d\mu _{t} \right )
\nonumber
\\
& \quad + \int \int _{0}^{t} P_{t-s}^{(t)} f \nabla P_{s}^{(t)}(\Phi _{t}'(g))
\cdot \nabla \dot{V}_{t} dsd\mu _{t}
\\
& \quad - \int \int _{0}^{t} [\nabla V_{t} \cdot \nabla \dot{V}_{t} -
\Delta \dot{V}_{t} ] P_{t-s}^{(t)} f P_{s}^{(t)} (\Phi _{t}'(g)) ds d
\mu _{t} .
\nonumber
\end{align}
We analyze each term separately.

Since $x\Phi _{t}'(x) \geq \Phi _{t}(x)$, it holds
$\int g \Phi _{t}'(g) d\mu _{t} \geq \int \Phi _{t}(g) d\mu _{t} = 1$.
Hence, if $N'(t) \geq 0$, the left hand side of the latter is bounded below
by $N'(t)$.

Since
$(\dot{V}_{t})_{-} \leq \frac{\delta _{t}}{4C_{t}} \left (|\nabla V_{t}|^{2}
- 2 \Delta V_{t}\right ) + F_{t}$, we can use the expansion of the square,
namely Inequality {\eqref{eq:Ubound}} with $f=\sqrt{\Phi _{t}(g)}$, to get
that
\begin{align*}
- \int \Phi _{t}(g) \dot{V}_{t} d\mu _{t} & \leq \int \Phi _{t}(g) (
\dot{V}_{t})_{-} d\mu _{t} \leq \frac{\delta _{t}}{C_{t}} \int
\frac{{\Phi _{t}'}^{2}(g)}{\Phi _{t}(g)} |\nabla g|^{2} d\mu _{t} + F_{t}
\\
& \leq \delta _{t} \int \Phi _{t}''(g) |\nabla g|^{2} d\mu _{t} + F_{t}
.
\end{align*}

Now assume first that assumption $(i)$ holds, namely that
$c_{t}=\|(\nabla V_{t} \cdot \nabla \dot{V}_{t} - \Delta \dot{V}_{t})_{-}
\|_{\infty }< \infty $. In that case we can proceed as in the proof of {\ref{th:main1}} to get
\begin{align*}
&\int \int _{0}^{t} P_{t-s}^{(t)} f \nabla P_{s}^{(t)}(\Phi _{t}'(g))
\cdot \nabla \dot{V}_{t} dsd\mu _{t} \\ 
& \quad - \int \int _{0}^{t} [\nabla V_{t}
\cdot \nabla \dot{V}_{t} - \Delta \dot{V}_{t} ] P_{t-s}^{(t)} f P_{s}^{(t)}
(\Phi _{t}'(g)) ds d\mu _{t}
\\
& \leq b_{t} \int \int _{0}^{t} P_{t-s}^{(t)} f |\nabla P_{s}^{(t)}(
\Phi _{t}'(g))| dsd\mu _{t} +c_{t} \int \int _{0}^{t} P_{t-s}^{(t)} f P_{s}^{(t)}
(\Phi _{t}'(g)) ds d\mu _{t}
\\
& \leq b_{t} N(t) \frac{1-e^{-\rho _{t} t}}{\rho _{t}} \int g \Phi _{t}''(g)
|\nabla g| d\mu _{t} + c_{t} tN(t) \int g \Phi _{t}'(g) d\mu _{t} .
\end{align*}
Using our assumption on $\Phi _{t}$ and
$uv \leq \frac{1}{2 \varepsilon }u^{2} + \frac{\varepsilon }{2}v^{2}$, the
latter is bounded above, for any $\varepsilon >0$, by
\begin{align*}
& b_{t} N(t) \frac{1-e^{-\rho _{t} t}}{\rho _{t}} \left (
\frac{1}{2\varepsilon } \int g^{2} \Phi _{t}''(g) d\mu _{t} +
\frac{\varepsilon }{2} \int \Phi _{t}''(g) |\nabla g|^{2} d\mu _{t}
\right ) \\
& \quad + c_{t}B_{t}tN(t) \int \Phi _{t}(g) d\mu _{t}
\\
& \leq b_{t} N(t) \frac{1-e^{-\rho _{t} t}}{\rho _{t}}
\frac{1}{2\varepsilon }(D_{t}+E_{t}) 
+ c_{t}B_{t}tN(t) \\
& \quad + \varepsilon
\frac{b_{t}N(t)(1-e^{-\rho _{t} t})}{2 \rho _{t}} \int \Phi _{t}''(g) |
\nabla g|^{2} d\mu _{t}.
\end{align*}
Choose $\varepsilon $ so that
$\varepsilon \frac{b_{t}(1-e^{-\rho _{t} t})}{2 \rho _{t}} = [1-
\delta -\bar{\rho }_{t}]/2$ so that, collecting the above computations together
with inequality {\eqref{eq:Fsob2}}, we can conclude that for any $t$ such
that $N'(t) \geq 0$,
\begin{equation*}
N'(t) \leq N(t) \left (F_{t} + \left ( b_{t}
\frac{1-e^{-\rho _{t} t}}{\rho _{t}} \right )^{2}
\frac{D_{t}+E_{t}}{2(1-\delta _{t} - \bar{\rho }_{t})} + c_{t}B_{t} t
\right )
\end{equation*}
from which the conclusion under assumption $(i)$ follows.

Now we turn to assumption $(ii)$. We need to bound the last two terms in {\eqref{eq:last2}}. Using Proposition~{\ref{prop:RZ}} with
$W:=W_{t}/b_{t}$ (observe that, since $\mu _{t}$ is a probability measure,
$b_{t} \neq 0$), it holds
\begin{align*}
&\int \int _{0}^{t} P_{t-s}^{(t)} f \nabla P_{s}^{(t)}(\Phi _{t}'(g))
\cdot \nabla \dot{V}_{t} dsd\mu _{t} \\
& \quad - \int \int _{0}^{t} [\nabla V_{t}
\cdot \nabla \dot{V}_{t} - \Delta \dot{V}_{t} ] P_{t-s}^{(t)} f P_{s}^{(t)}
(\Phi _{t}'(g)) ds d\mu _{t}
\\
& \leq b_{t} \int \int _{0}^{t} P_{t-s}^{(t)} f \left ( |\nabla P_{s}^{(t)}(
\Phi _{t}'(g))| + \frac{W_{t}}{b_{t}} P_{s}^{(t)} (\Phi _{t}'(g))
\right ) ds d\mu _{t}
\\
& \leq b_{t} \int \int _{0}^{t} e^{(c_{s} - \rho _{s})s} P_{t-s}^{(t)}
f P_{s}^{(t)} \left ( |\nabla \Phi '(g)| + \frac{W_{t}}{b_{t}} \Phi '(g)
\right )
\\
& = N(t) \int _{0}^{t} e^{(c_{s} - \rho _{s})s} ds \left ( b_{t}
\int g\Phi _{t}''(g) |\nabla g| d\mu _{t} + \int W_{t} g \Phi _{t}'(g)
d\mu _{t} \right )
\end{align*}
where we used the reversibility in the last inequality. For the first term
in the right hand side of the latter, we proceed as for assumption
$(i)$. Namely, it holds for all $\varepsilon >0$
\begin{align*}
\int g\Phi _{t}''(g) |\nabla g| d\mu _{t} & \leq
\frac{1}{2 \varepsilon }\int g^{2} \Phi _{t}''(g) d\mu _{t} +
\frac{\varepsilon }{2} \int \Phi _{t}''(g) |\nabla g|^{2} d\mu _{t}
\\
& \leq \frac{1}{2 \varepsilon } (D_{t}+E_{t}) + \frac{\varepsilon }{2}
\int \Phi _{t}''(g) |\nabla g|^{2} d\mu _{t} .
\end{align*}
For the second term, we use the expansion of the square (inequality {\eqref{eq:Ubound}} with $f=\sqrt{\Phi _{t}(g)}$) to get that
\begin{align*}
\int W_{t} g \Phi _{t}'(g) d\mu _{t} & \leq B_{t} \int W_{t} \Phi _{t}(g)
d\mu _{t} \\
& \leq 
\frac{\delta '}{4C_{t}} \int \Phi _{t}(g) \left (|
\nabla V_{t}|^{2} - 2 \Delta V_{t}\right )d\mu _{t} + B_{t} F'_{t}
\\
& \leq \frac{\delta '}{4C_{t}} \int
\frac{\Phi _{t}'(g)^{2}}{\Phi _{t}(g)}|\nabla g|^{2} d\mu _{t} + B_{t}
F'_{t} \\
& \leq 
\delta ' \int \Phi _{t}''(g)^{2}\|\nabla g|^{2} d\mu _{t} +
B_{t} F'_{t} .
\end{align*}
Summarizing, under assumption $(ii)$, when $N'(t) > 0$, we obtain
\begin{align*}
\frac{N'(t)}{N(t)} & \leq \int \dot{\Phi }_{t}(g) d\mu _{t} \\
& \quad 
- \Big( 1-
\delta _{t} - \frac{\varepsilon }{2} \int _{0}^{t} e^{(c'_{s}-\rho _{s})s}ds + \delta '_{t} \int _{0}^{t} e^{(c'_{s}-\rho _{s})s}ds \Big) \int
\phi _{t}''(g) |\nabla g|^{2} d\mu _{t}
\\
& \quad {}+ \int _{0}^{t} e^{(c'_{s}-\rho _{s})s}ds \left (
\frac{b_{t}(D_{t}+E_{t})}{2\varepsilon } + B_{t}F'_{t} + F_{t} \right )
\\
& \leq \int _{0}^{t} e^{(c'_{s}-\rho _{s})s}ds \left (
\frac{b_{t}(D_{t}+E_{t})}{2\varepsilon } + B_{t}F'_{t} + F_{t} \right )
\qquad \text{ (thanks to {\eqref{eq:Fsob2}})}
\\
& = \left ( \int _{0}^{t} e^{(c'_{s}-\rho _{s})s}ds \right )^{2}
\frac{b_{t}(D_{t}+E_{t})}{2(1-\delta _{t} - \bar{\rho }_{t} - \delta '_{t} \int _{0}^{t} e^{(c'_{s}-\rho _{s})s}ds)}
\\
&\quad {}
+ \int _{0}^{t} e^{(c'_{s}-\rho _{s})s}ds \left ( B_{t}F'_{t} + F_{t}
\right )
\end{align*}
for $\varepsilon $ so that
$\frac{\varepsilon }{2} \int _{0}^{t} e^{(c'_{s}-\rho _{s})s}ds =
\frac{1-\delta _{t} - \bar{\rho }_{t} - \delta '_{t} \int _{0}^{t} e^{(c'_{s}-\rho _{s})s}ds }{2}$.
The desired conclusion follows.
\end{proof}

In the proof of {\ref{th:main2}} we used the following results borrowed
from \cite{RZ}.

\begin{proposition}
\label{prop:RZ}
Let $L=\Delta - \nabla U \cdot \nabla $, on $\mathbb{R}^{n}$, and denote
by $(P_{t})_{t \geq 0}$ its associated semi-group. Assume that
$U \colon \mathbb{R}^{n} \to \mathbb{R}$ is smooth enough and satisfies
$\int e^{-U} = 1$ so that $\mu (dx)=e^{-U(x)}dx$ is a probability measure
on $\mathbb{R}^{n}$ and $\mathrm{Hess}(U) \geq \rho $ (as a matrix) for
some $\rho \in \mathbb{R}$. Let
$W \colon \mathbb{R} \to \mathbb{R}_{+}$ be such that
$c:= \max \left (2 \| |\nabla W |\|_{\infty }, \sup _{x:W(x) \neq 0}
\left ( \frac{LW}{W} - \rho \right )_{-} \right ) < \infty $. Then, for
all $f$ non negative
\begin{equation*}
|\nabla P_{t}f| + W P_{t} f \leq e^{(c-\rho )t} P_{t} \left ( |
\nabla f| + W f \right ) .
\end{equation*}
\end{proposition}




\appendix{}
\label{appA}
Here we discuss the following formula
\begin{equation*}
\partial _{t} P_{t}^{(t)}f = L_{t} f + \int _{0}^{t} \left ( e^{\tau L_{(t)}}
\dot{L_{t}}e^{ (t-\tau ) L_{(t)} }f\right ) d\tau .
\end{equation*}
Note that
\begin{equation*}
\begin{split} \partial _{t} P_{t}^{(t)}f &= \lim _{s\to 0}\frac{1}{s}
\left (P_{t+s}^{(t+s)}f - P_{t}^{(t)}f\right )
\\
&= \lim _{s\to 0}\frac{1}{s}\left (P_{t}^{(t+s)}f -P_{t}^{(t)}f
\right ) + \lim _{s\to 0}\frac{1}{s}\left ( P_{ t+s }^{(t+s)}f - P_{t}^{(t+s)}f
\right )
\end{split}
\end{equation*}
For the first term on the right hand side we have
\begin{equation*}
P_{t}^{(t+s)}f - P_{t}^{(t)}f = \int _{0}^{t} \left ( e^{\tau L_{(t+s)}}
\left (L_{(t+s)}-L_{(t)}\right )e^{ (t-\tau ) L_{(t)} }f\right ) d
\tau
\end{equation*}
provided $e^{(t-\tau ) L_{(t)} }f$ is in the domain of
$L_{(t+s)}-L_{(t)}$ for every sufficiently small $s$ and all
$\tau \in [0,t]$. Hence if the limit
\begin{equation*}
\lim _{s\to 0}\frac{1}{s} \left (L_{(t+s)}-L_{(t)}\right )e^{(t-\tau )
L_{(t)} } f \equiv \dot{L_{t}}e^{(t-\tau ) L_{(t)} } f
\end{equation*}
is well defined, we have
\begin{equation*}
\lim _{s\to 0}\frac{1}{s}\left (P_{t}^{(t+s)}f -P_{t}^{(t)}f\right ) =
\int _{0}^{t} \left ( e^{\tau L_{(t)}} \dot{L_{t}}e^{ (t-\tau ) L_{(t)}
}f\right ) d\tau
\end{equation*}
On the other hand
\begin{equation*}
P_{ t+s }^{(t+s)}f - P_{t}^{(t+s)}f = L_{(t+s)} \int _{0}^{s} e^{ (t+
\tau ) L_{(t+s)} }f d\tau
\end{equation*}
is well defined for $C_{0}$-semigroup and for $f$ in the domain of
$L_{t}$we have
\begin{equation*}
\lim _{s\to 0}\frac{1}{s}\left ( P_{ t+s }^{(t+s)}f - P_{t}^{(t+s)}f
\right ) = L_{t} f .
\end{equation*}
Combining all the above yields
\begin{equation*}
\partial _{t} P_{t}^{(t)}f = L_{t} f + \int _{0}^{t} \left ( e^{\tau L_{(t)}}
\dot{L_{t}}e^{ (t-\tau ) L_{(t)} }f\right ) d\tau .
\end{equation*}
%
%
%

\label{bibliography}
\bibliographystyle{plain}

\def\cprime{$'$}

\end{document}